\documentclass[12pt, a4paper]{article}

\usepackage{amsfonts, latexsym, amssymb, amsthm, amsmath, mathrsfs, esint, color}
\usepackage[latin1]{inputenc}
\usepackage{enumerate}
\usepackage{amsmath}
\usepackage{amsthm}
\usepackage{amssymb}
\usepackage{amscd}
\usepackage{enumerate}
\usepackage{pstricks}
\usepackage{pst-node}
\usepackage{esint}
\usepackage{mathtools}
\usepackage{graphicx}
\usepackage{relsize}
\usepackage{hyperref}
\usepackage{comment}

\def\d{\partial}

\newcommand{\R}{\mathbb{R}}

\def\mr{\mathbf{R}}
\def\mz{\mathbf{Z}}

\renewcommand{\o}{\tilde}

\renewcommand{\d}{\partial}

\renewcommand{\o}{\tilde}

\newcommand{\SO}[1]{\operatorname{SO}(#1)}
\newcommand\id{I}
\newcommand\sym{\operatorname{sym}}
\newcommand\skw{\operatorname{skw}}

\newcommand\dist{\operatorname{dist}}

\DeclareMathOperator{\cof}{cof}

\newcommand\loc{\textrm{loc}}

\def\XXint#1#2#3{{\setbox0=\hbox{$#1{#2#3}{\int}$}
     \vcenter{\hbox{$#2#3$}}\kern-.5\wd0}}

\DeclareMathOperator{\secf}{II}

\newcounter{bei}

\renewcommand{\o}{\overline}

\renewcommand{\div}{\mathrm{div}\,}

\newtheorem{theorem}{Theorem}[section]
\newtheorem{lemma}[theorem]{Lemma}
\newtheorem{proposition}[theorem]{Proposition}

\newtheorem{assumption}[theorem]{Assumption}
\newtheorem{definition}[theorem]{Definition}
\newtheorem{remark}[theorem]{Remark}
\theoremstyle{remark}
\newtheorem{example}[theorem]{Example}

\usepackage{mathtools} 
\usepackage[normalem]{ulem}
\newcommand{\randomspace}{\Omega}
\newcommand{\pspace}{S}
\newcommand{\randomelement}{\omega}

\newcommand{\tomega}{{\widetilde\randomelement}}
\newcommand{\randommeasure}{P}
\newcommand{\td}{\mathrm{d}}
\newcommand{\drandommeasure}{\mathrm{d}\randommeasure(\randomelement)}

\newcommand{\wtwoscale}{\xrightharpoonup{2}}
\newcommand{\wctwoscale}{\xrightharpoonup{2c}}
\newcommand{\stwoscale}{\xrightarrow{2}}
\newcommand{\weakly}{\rightharpoonup}

\newcommand{\stochsobolev}{\mathcal W}
\newcommand{\stochsmooth}{\mathcal C^\infty}
\newcommand{\twodomain}{S}

\newcommand{\eps}{\varepsilon}

\newcommand{\testfunctions} { \mathcal D}
\newcommand{\distributions}{\mathcal D'}

\newcommand{\adh}{\mathrm{adh}}
\newcommand{\de}{u}

\newcommand{\weakstar}{\stackrel{\star}{\weakly}}

\let\intersect\cap
\let\bigunion\bigcup

\DeclarePairedDelimiter\abs{\lvert}{\rvert}
\DeclarePairedDelimiter\norm{\lVert}{\rVert}
\DeclarePairedDelimiter\scalar{\langle}{\rangle}
\DeclarePairedDelimiter\set{\{}{\}}
\DeclarePairedDelimiter\paren{(}{)}
\DeclarePairedDelimiter\brackets{[}{]}

\newcommand\restrict[2]{{
  \left.\kern-\nulldelimiterspace 
  #1 
  \vphantom{\big|} 
  \right|_{#2} 
  }}

\usepackage{color}
\usepackage{ cancel }
\usepackage[color]{changebar}

\newcommand{\matt}[1]{{\color{red}{#1}} }

\cbcolor{red}
\begin{document}
\title{Stochastic homogenization of the bending plate model}

\author{
Peter Hornung
\\
FB Mathematik, TU Dresden
\\
01062 Dresden (Germany)
%
\\
\\
Matth\"aus Pawelczyk
\\
FB Mathematik, TU Dresden
\\
01062 Dresden (Germany)
%
\\
\\
Igor Vel\v ci\'c
\\
University of Zagreb, \\Faculty of Electrical Engineering and Computing
\\
Unska 3, Zagreb (Croatia)
}

\date{}

\maketitle
\begin{abstract}
We use the notion of stochastic two-scale convergence introduced in \cite{zhikov1} to solve the problem of stochastic homogenization of the elastic plate in the bending regime. 
\end{abstract}
\vspace{10pt}

\noindent {\bf Keywords:}
elasticity, dimension reduction, stochastic homogenization, stochastic
two-scale convergence.

\section{Introduction}

The problem of rigorously deriving a two-dimensional model approximating a three-dimensional (nonlinear) elastic plate with 
very small thickness was long outstanding. It was finally done in \cite{FJM-02} in terms of $\Gamma$-convergence after establishing the geometric rigidity estimate. With this estimate they further managed in \cite{FJM-06} 
to derive a multitude of related models. Their results have since been generalized in various directions, e.g., different dimensions involved (e.g., \cite{MoraMueller-04}), convergence of equilibria instead of convergence of minimizer as $\Gamma$-convergence yields (e.g., \cite{MuellerPakzad-08}, \cite{MoraMueller-08}), or an inhomogeneous plate (e.g.,\cite{Horneuvel12}, \cite{Neukamm-Olbermann}, \cite{Vel13}). This paper falls in the last category.

We consider a thin plate with a fine microstructure on the midplane, extended constantly in normal direction. A very similar problem was studied in \cite{Horneuvel12}, where the microstructure was assumed to be periodic, while we consider more general random materials and recover their main results as a special case. Another interesting generalization of the periodic case was given in \cite{BDF15}, where the microstructure was allowed to oscillate on two different scales $\eps_1(h)$ and $\eps_2(h)$, where the `coarser' structure dominates the homogenization effect.

Already in the periodic case it was seen that the homogenization and the dimension reduction 
interact non-trivially with each other. To be more precise let $h > 0$ denote the thickness of the plate,
and $\eps(h)$ the `fineness' of the microstructure, e.g., in the periodic case the length of a periodic cell, at thickness $h$ with $\eps(h) \to 0$ if $h\to0$ and assume $\gamma = \lim_{h\to0}h\eps^{-1}(h)  \in [0,\infty]$ exists. One might imagine the case $\gamma = \infty$ corresponds to the situation, where we apply purely dimension reduction to an already homogenous plate, while one could expect $\gamma = 0$ to be the case where a $2$D plate is homogenized; the latter, however, is wrong at least in the plate scaling as comparing the results obtained in \cite{Vel13} and \cite{Neukamm-Olbermann} shows.
This intuition, however, holds true for the von K\'arm\'an plate \cite{NeuVel-13}.
 The intermediate case $0 < \gamma < \infty$ corresponds to the case, where both effects strongly interact; in some sense thus the most interesting case.

With a periodic microstructure in \cite{Horneuvel12} the range $\gamma \in (0, \infty]$ was covered, excluding the $0$ entirely.
The methods developed in \cite{Vel13} allows the treatment, at least partially, of the case $\gamma= 0$. Only partially, since we have to assume the microstructure is still sufficiently strong `homogenizing', i.e.\ $h \gg (\eps(h))^2$. The stochastic homogenization incorporates periodic setting, almost periodic setting, but also some completely non-periodic examples (see the Example \ref{borism3} below). Since it is possible to have the situation where periodicity is completely destroyed and since we are not able to treat all cases of the periodic homogenization, we find that it is important to establish the result on the stochastic homogenization of the bending plate. 

As in \cite{Horneuvel12} we make heavy use of two-scale convergence. The first generalization to the stochastic setting of two-scale convergence was done in \cite{mikelic},
which is too crude to recover the information on the limit material. An alternative was introduced by \cite{zhikov1}, which is more flexible, 
and which we will use.

Compared to the periodic setting additional difficulties arise in identifying the two-scale limit
sufficiently enough to recover the effective properties of the material. In fact we are not to recover all the limits as done in \cite{Horneuvel12}, specifically we are not able to use the notion of the oscillatory convergence. 

To cope with this, we use methods developed in \cite{Vel13} and make use of further cancelation effects  (see Lemma~\ref{lem:(31)} and Lemma~\ref{lem:(33)}).
   Furthermore, the precise relationship between solenoids and potential fields 
were not known, when the differential operators $\div$ and $\nabla$ were not either purely classical, or purely stochastic derivatives, but mixtures between them. In the appendix we recall previous results for the purely stochastic case, and prove the Helmholtz-decomposition for the mixed one.
This allows sufficient identification of the two-scale limits by testing with solenoids,
a subclass of functions used in the oscillatory convergence, introduced in \cite{Horneuvel12}.

For simplicity, we will state and prove the case $ \gamma \in (0,\infty)$, but the other cases covered  in \cite{Horneuvel12,Vel13} can be proved analogously. This includes the case $\gamma= \infty$ as well as $\gamma=0$, under the additional assumption that $\varepsilon(h)^2 \ll h \ll \varepsilon(h)$. 

Without the notion of stochastic two-scale convergence we are not able 
to solve the problem;
the usual approach for stochastic homogenization in the context of 
calculus of variation however does not necessarily need stochastic two-scale 
convergence, cf. \cite{dalmasomodica1,dalmasomodica2} for the convex 
case and \cite{gloria1} for the non-convex case. The main claims of the paper are given in Theorem~\ref{theoremlowbou} and Theorem~\ref{theoremupperbound}.
\subsection*{Notation}
 The inner product of the Hilbert space $H$ for the vectors $a,b \in H$ is denoted by $\scalar{a,b}_H$. If $H=\mathbf{R}^n$ it is also denoted by $a \cdot b$. 
By $\iota: \mathbf{R}^{2 \times 2} \to \mathbf{R}^{3 \times 3}$ we denote the natural inclusion
\[
  \iota (G)=\sum_{\alpha,\beta=1,2}  G_{\alpha \beta} e_{\alpha} \otimes e_{\beta}.
\]
For a matrix $M \in \mathbf{R}^{n \times n}$ we denote its transpose by $M^T$ and by $\cof M$ its cofactor matrix. 
By $\id_{n\times n}$ we denote the identity matrix in $\mathbf{R}^{n \times n}$. 
By $\nabla'$ we denote the gradient with respect to the first two variables $\nabla'f=(\partial_1 f, \partial_2 f)$ and similar $x' = (x_1, x_2) \in \mr^2$ for $x\in \mr^3$. 
For $h > 0$ we furthermore denote by $\nabla_h$ the scaled gradient, given by $\nabla_h f = ( \nabla' f, \frac 1 h \partial_3 f)$.
By $\mathbf{R}^{n \times n}_{\sym}$ we denote the space of symmetric matrices, while by $\mathbf{R}^{n \times n}_{\skw }$ we denote the space of antisymmetric matrices.
 For a normed space $X$ and $S \subset X$ we denote by $\adh_X S$ the closure of the set $S$ in the norm defined on the space $X$.

\section{Stochastic two-scale convergence} 
\subsection{Probability framework} 
Let $(\Omega, \mathcal{F},P)$ be a probability space. We will assume that $\mathcal{F}$ is countably generated which implies that the spaces $L^p(\Omega)$, for $p \in [1,\infty) $, are separable.
By $S$ we will denote the domain in $\mathbf{R}^n$. 
With $I$ we denote the interval $I=[-\tfrac{1}{2}, \tfrac{1}{2}]$.
\begin{definition}\label{defgroup}
	A family $(T_x)_{x \in \mathbf{R}^n}$ of measurable bijective mappings $T_x:\Omega \to \Omega$ on a probability space $(\Omega, \mathcal{F}, P)$ is called a dynamical system on $\Omega$ with respect to $P$ if
	\begin{enumerate}
		\item $T_x \circ T_y=T_{x+y}$;
		\item $P(T_{x} F )=P(F)$, $\forall x \in \mathbf{R}^n$, $F \in \mathcal{F}$;
		\item $\mathcal{T}: \Omega \times \mathbf{R}^n  \to \Omega$, $(\omega,x)\to T_x (\omega)$ is measurable (for the standard $\sigma$-algebra on the product space, where on $\mathbf{R}^n$ we take the Lebesgue $\sigma$-algebra).
	\end{enumerate}
\end{definition}
We define the notion of ergodicity for the dynamical system.
\begin{definition}\label{defergodic}
A dynamical system is called ergodic if one of the following equivalent condion is fullfilled 
\begin{enumerate}
	\item $f$ measurable, $f(\omega)=f(T_x \omega), \ \forall x \in \mathbf{R}^n, \textrm{ a.e. } \omega \in \Omega \implies f(\omega)=\textrm{const. for $P$-a.e. } \omega \in \Omega$.
        \item\label{def:ergodicityb} $\brackets[\Big]{\forall x \in \mathbf{R}^n,   P((T_x B \cup B) \backslash (T_x B \cap B))=0} \implies P(B) \in \{0,1\}$.  
\end{enumerate}	
\end{definition} 
\begin{remark}\label{rem:ergo-weak}
Note that for the condition \ref{def:ergodicityb} the implication $P(B) \in \set { 0,1}$ has to hold, if the symmetric difference between $T_x B$ and $B$ is a null set.
It can be shown (e.g., \cite{cornfeld}), that ergodicity is also equivalent if a priori only the weaker implication
\[
  \brackets[\Big]{\forall x \in \mr^N, \quad T_x B = B} \implies P(B) \in \set{ 0,1}
\]
holds. This formulation will however only be used in the appendix to show that the product of an ergodic system with a periodic one is once more ergodic.
\end{remark}

For $f \in L^p(\Omega)$ write $f(\omega,x):=f(T_x \omega)$, defining the realization $f \in L^p_{\textrm{loc}}(\mathbf{R}^n, L^p(\Omega))$. 
On $L^2(\Omega)$ we can define the unitary action
\[ U(x)f=f \circ T_x, \quad \forall f \in L^2(\Omega). \]
It  can be shown that $a,b,c$ of Definition \ref{defgroup} imply that this is a strongly continuous group (see \cite{zhikov2}).
We define the operator $D_i$ as the infinitesimal generator of the unitary group $U_{x_i}$. 
This means that
\[
  D_i f(\omega)=\lim_{x_i \to 0}  \frac{f(T_{x_i}\omega)-f(\omega) }{x_i},    
\]
where the limit is taken in $L^2$ sense. 
Also we have that  $i D_1,\dots, i D_n$ are commuting, self-adjoint, closed, and densely defined linear operators on the separable Hilbert space $L^2 (\Omega)$. 
The domain $\mathcal{D}_i(\Omega)$ of such an operator is given by the set of $L^2$ functions for which the limit exists. We denote by $W^{1,2}(\Omega)$ the set
\[
   W^{1,2}(\Omega):=\mathcal{D}_1(\Omega) \cap \dots \cap \mathcal{D}_n(\Omega)
\]
and similarly
\begin{align*}
	W^{k,2} (\Omega)&=\{f \in L^2(\Omega): D_1^{\alpha_1}\dots D_n^{\alpha_n} f \in L^2(\Omega),\; \alpha_1+\cdots +\alpha_n=k\};\\
	W^{\infty,2} (\Omega)&= \bigcap_{k \in \mathbf{N}} W^{k,2}(\Omega).
\end{align*}
By the standard semigroup property it can be shown that $W^{\infty,2}(\Omega)$ is dense in $L^2(\Omega)$. 
We also define the space
\[
   \mathcal{C}^{\infty} (\Omega)= \set[\big] {f \in W^{\infty,2} (\Omega): \forall (\alpha_1,\dots, \alpha_n) \in \mathbf{N}_0^n, \quad D_1^{\alpha_1} \dots D_n^{\alpha_n} f \in L^{\infty} (\Omega) }. 
\]
By the smoothening procedure explained below it can be shown that $\mathcal{C}^{\infty} (\Omega)$ is dense in $L^p(\Omega)$ for any $p \in [1,\infty)$ as well as in $W^{k,2}(\Omega)$ for any $k$. 
Notice that $D_i f$, due to the closedness property of the infinitesimal generator, can be equivalently defined as the function that satisfies the property 
\[
   \int_{\Omega} D_i f g =-\int_{\Omega } f D_i g,\quad \forall g \in \mathcal{C}^{\infty}(\Omega).
\]
After identifying $f \in W^{1,2} (\Omega)$ with its realization, one can show that 
\begin{equation}\begin{aligned}\label{eq:sobolevrealizations}
W^{1,2} (\Omega) &= \{ f \in W^{1,2}_{\textrm{loc}} (\mathbf{R}^n, L^2(\Omega)): f(x+y,\omega)= f(x, T_y \omega), \quad \forall x,y, \text{for a.e. }  \omega \}\\ 
&= \{ f \in C^1(\mathbf{R}^n, L^2(\Omega)): f(x+y,\omega)= f(x, T_y \omega), \quad  \forall x,y,\text{for a.e. }  \omega  \}. 
\end{aligned}\end{equation}
A proof of this fact can be found in \cite{gloria1}[Lemma~A.7].

As in \cite{zhikov2} we define a smoothening operator. For $\varphi \in L^{\infty} (\Omega)$ and $K \in C^{\infty}_0 (\mathbf{R}^n)$ even, i.e. $K(x) = K(-x)$ for all $x \in \mr^n$, we set
\[
  \paren{\varphi*K}(\omega):=\int_{\mathbf{R}^n} \varphi(T_x\omega) K(x) \td x ,\quad \omega \in \Omega.
\]
It is easily seen that $\varphi \mapsto \varphi *K$ is well defined and continuous from $L^2(\Omega)$ to $L^2(\Omega)$. 
By using this mollifier one can show that there exists a countable dense subset of $L^2(\Omega)$ and $W^{1,2} (\Omega)$ (see \cite{mikelic}). 
Following \cite{sango1} we denote by $\norm \cdot_{\#,k,2}$ the seminorm on $\mathcal{C}^{\infty} (\Omega)$ given by 
\[
  \norm u_{\#,k, 2}^2=\sum_{\alpha \in \mathbb N^n, \abs { \alpha } = k} \norm {D^\alpha u}^2_{L^2}.
\]
By $\mathcal{W}^{k,2}(\Omega)$ we denote the completion of $\mathcal{C}^{\infty}(\Omega)$ with respect to the seminorm $\norm \cdot_{\#,k,2}$.
The gradient operator $\nabla_{\omega}=(D_{1},\dots, D_n)$ and $\div_{\omega}= \nabla_{\omega} \cdot$ operator extend by continuity uniquely to mappings from $\mathcal{W}^{1,2}(\Omega)$ to $L^2(\Omega,\mathbf{R}^n)$, respectively $\mathcal{W}^{1,2}(\Omega,\mathbf{R}^n)$ to $L^2(\Omega)$. 
By the density argument it is easily seen that $\mathcal{W}^{1,2}(\Omega)$ is also the completion of $W^{1,2}(\Omega)$ in $\norm \cdot_{\#,1,2}$ seminorm. 
We also define $ \stochsobolev_{\sym}^{1,2}(\Omega, \mr^n)$ as the completion of $\stochsmooth(\Omega, \mr^n)$ with respect to the seminorm $\norm \cdot_{\#,\sym,2,n}$
defined by
\[
   \|b\|_{\#,\sym,2,n}= \|\sym \nabla b\|_{L^2}, \forall b \in \stochsmooth(\Omega, \mr^n).
\]

\subsection{Definition and basic properties}

The key property of ergodic systems is the ergodic theorem, due to Birkhoff:
\begin{theorem}[Ergodic theorem] \label{thmergodic}
	Let  $(\Omega,\mathcal{F}, P)$ be a probability space with an ergodic dynamical system $(T_x)_{x \in \mathbf{R}^n}$ on $\Omega$. Let $f \in L^1(\Omega)$ be a function and $A \subset \mathbf{R}^n$ be a bounded open set. Then for $P$-a.e. $\omega \in \Omega$ we have
        \begin{equation}\begin{aligned}\label{eq:birkhoff}
        \lim_{\varepsilon \to 0} \int_A f(T_{\varepsilon^{-1} x}\omega) \td x= |A| \int_{\Omega} f(\omega) \drandommeasure.
        \end{aligned}\end{equation}
	Furthermore, for every $f \in L^p(\Omega)$, $1 \leq p \leq \infty$, and a.e. $\omega \in \Omega$, the function $f (\omega,x)= f(T_x \omega)$ satisfies $f(\omega,\cdot) \in L^p_{\textrm{loc}} (\mathbf{R}^n)$. For $p < \infty$ we have $f(\omega,\cdot/ \varepsilon)=f(T_{ \cdot/ \varepsilon} \omega ) \rightharpoonup \int_{\Omega} f \td P$ weakly in $L^p_{\textrm{loc}}(\mathbf{R}^n )$ as $\varepsilon \to 0$. 
\end{theorem}
The elements $\omega$ such that \eqref{eq:birkhoff} holds for every $f \in L^1(\Omega)$ are called typical elements,
the corresponding trajectories $(T_x \omega)_{x \in \mr^N}$ are called typical trajectories. Note that the separability of $L^1(\Omega)$ implies that almost every $\omega \in \Omega$ is typical. 

In the following we denote by $S \subset \mathbf{R}^n$ a bounded domain, if not otherwise stated. For vector spaces $V_1, V_2$ by $V_1 \otimes V_2$ we denote the usual tensor product of the spaces $V_1,V_2$. We define the following notion of stochastic two-scale convergence, a slight variation of the definition given in \cite{zhikov1}. We will stay in the Hilbert setting, since it suffices for our analysis.

\begin{definition}\label{definicija1}
Let $(T_{x}\randomelement)_{x \in \R^n}$ be a typical trajectory and $(v^\eps)$ a bounded sequence
in $L^2(S)$. We say that $(v^\eps)$ {stochastically weakly two-scale converges} to 
$v \in L^2(\randomspace \times\pspace)$ w.r.t.\ $\randomelement$ and we write
  $v^\eps \wtwoscale v$ if 
\[
  \lim_{\eps \downarrow 0} \int_{\pspace} v^\eps(x)g(T_{\eps^{-1} x} \randomelement,x) \td x
  = \int_{\randomspace} \int_\pspace v(\randomelement,x) g( \randomelement,x) \td x\, \drandommeasure
\]
for all $g \in \stochsmooth(\randomspace)\otimes C^\infty_0(\pspace)$.

If additionally 
\[
  \norm  { v^\eps}_{L^2(S)} \to \norm { v}_{L^2(\Omega \times S)},
\]
holds, then we say $(v^\eps)$ strongly two-scale converges to $v$ and write $v^\eps \stwoscale v$.
\end{definition}
For  vector-valued functions we define the convergence componentwise.

\begin{remark} 
The convergence of the sequence $(v^{\varepsilon})$ is defined along a typical trajectory and thus the limit can also be $\omega$-dependent. We don't write this dependence, since we will always look at the problem on a typical trajectory  (which can be imagined to be fixed).
\end{remark} 	
\begin{remark}
Note that the two-scale limit is defined on the whole space $\Omega \times S$, and also by density we can extend the space of test functions $g$ to $ L^\infty(\Omega)\otimes L^2(S)$. 
\end{remark}
\begin{remark}\label{borism7} 
Since we will assume that our material oscillates in the in-plane direction on the domain $\pspace \times I$ we will often use the notion of in-plane two-scale convergence for the sequence that is bounded in $L^2(S \times I)$, i.e. we  say that bounded sequence in $L^2(\pspace \times I)$, $(v^\eps)$ {stochastically weakly two-scale converges} to 
$v \in L^2( \randomspace\times \pspace \times I)$ w.r.t.\ $\randomelement$ and we write
$v^\eps \wtwoscale v$ if 
\[
\lim_{\eps \downarrow 0} \int_{\pspace \times I} v^\eps(x)g(T_{\eps^{-1} x'}\randomelement,x) \td x
= \int_{\randomspace} \int_{\pspace \times I} v( \randomelement,x) g(\randomelement,x) \td x\, \drandommeasure
\]
for all $g \in \stochsmooth(\randomspace) \otimes C^\infty_0(\pspace \times I)$.
All the properties of the previous stochastic two-scale convergence remain valid for this variation as well. 
\end{remark}

Sometimes we will make the decomposition for the two-scale limit $v$
$$ v(\omega,x)=\int_{\Omega} v(\omega,x)\drandommeasure+\left(v(\omega,x)-\int_{\Omega} v(\omega,x)\drandommeasure      \right),$$
separating the weak limit from the oscillatory part. We will then write 
$$v^{\eps} \wctwoscale  v(\omega,x)-\int_{\Omega} v(\omega,x)\drandommeasure.      $$
\begin{proposition}[Compactness]\label{lem:compactness}
Let $(v^\eps)$ be a bounded sequence in $L^2(S)$.
Then there exists a subsequence (not relabeled) and $v \in L^2(\randomspace \times S)$ 
such that $v^{\eps} \wtwoscale v$.
\end{proposition}
A proof can be found in \cite{zhikov1}[Lemma~5.1].

 The following proposition states the compatibilty of strong convergent sequences with weak two-scale convergent sequences.
\begin{proposition}\label{propositionreasonable} 
\begin{enumerate}
	\item If $(u^\eps) \subset L^2(S)$ is a bounded sequence with $u^\eps \to u$ in $L^2(S)$ for some $u \in L^2(S)$, then $u^\eps \wtwoscale u$.
	\item If $(v^\eps) \subset L^{\infty} (S)$ is uniformly bounded by a constant and $v^\eps \to v$ strongly in $L^1(S)$ for some $v \in L^\infty(S)$, and if $(u^\eps )$ is bounded in $L^2(S)$ with  $u^\eps \wtwoscale u$ for some $u \in L^2( \Omega \times S)$, then we have that $v^\eps u^\eps \wtwoscale vu$.    
\end{enumerate}
\end{proposition}
The proof is straightforward. The next lemma is useful to prove the following Lemma~\ref{lem:twoscalegradient}, which gives us the form of stochastic two-scale limits of gradients.

\begin{lemma} \label{prva}
Let $f \in (L^2 (\Omega))^n$ be such that
\[
  \int_{\Omega} f\cdot g=0, \quad \forall g \in \mathcal{C}^{\infty} (\Omega,\mr^n) \textrm{ satisfying } \div_{\omega} g=0.
\]
Then there exists $\psi \in \mathcal{W}^{1,2} (\Omega)$ such that 
$f=\nabla_{\omega} \psi$.
\end{lemma}
\begin{proof}
It is an immediate consequence of Theorem~\ref{thm:firstordercomp}.
\end{proof} 	
\begin{lemma}\label{lem:twoscalegradient}
Let $(u^\eps)$ be a bounded sequence in $W^{1,2}(S)$. Then on a subsequence (not relabeled) 
$u^\eps \weakly u^0$ in $W^{1,2}(S)$ and there exists $u^1 \in L^2( S, \stochsobolev^{1,2}(\randomspace))$ 
such that
\begin{align*}	
  \nabla u^\eps \wtwoscale \nabla u^0 + \nabla_\omega u^1\,.
\end{align*}
\end{lemma}
\begin{proof}
The statement follows immediately from the previous lemma in the same way as in the periodic case (see \cite{Allaire-02}). 
\end{proof}
		
Similar results hold for second gradients, at least for $n=2$ which the next two lemmas prove. 
\begin{lemma} \label{lem:orthosecondgradient}
Let $f \in L^2 (\Omega, \mr^{2 \times 2}_{\sym})$ be such that
\[
  \int_{\Omega} f\cdot \cof \nabla_{\omega}  g=0, \quad \forall g \in \mathcal{C}^{\infty} (\Omega,\mr^2).
\]
Then there exists $\psi \in \mathcal{W}^{2,2} (\Omega)$ such that 
$f=\nabla^2_{\omega} \psi$.
\end{lemma}
\begin{proof}
It is an immediate consequence of Theorem~\ref{thm:secorderdecomp}.
\end{proof} 	
\begin{lemma}\label{lem:twoscalesecondgradient}
Let $(u^\eps)$ be a bounded sequence in $W^{2,2}(S)$, where $S \subset \mathbf{R}^2$ a bounded domain. Then on a subsequence (not relabeled) 
$u^\eps \weakly u^0$ in $W^{2,2}(S)$ and there exists $u^1 \in L^2( S, \stochsobolev^{2,2}(\randomspace))$ 
such that
\begin{align*}	
  \nabla^2 u^\eps \wtwoscale \nabla^2 u^0 + \nabla^2_\omega u^1\,.
\end{align*}
\end{lemma}
\begin{proof}
By Prop.~\ref{lem:compactness} there exists $f \in L^2(\Omega \times S, \mr^{2\times2})$ and a subsequence with
\[
  \nabla^2 u^\eps \wtwoscale \nabla^2 u^0 + f.
\]
Since $ \nabla^2 u^\eps - \nabla^2 u^0 \in\mr^{2\times 2}_{\sym}$ almost everywhere on $S$, we get $f \in \mr^{2\times2}_{\sym}$ almost everywhere on $\Omega \times S$. Thus by Lemma~\ref{lem:orthosecondgradient} it suffices to show that  for almost every  $x \in S$ we have
\[
  \int_{\Omega} f( x,\omega )\cdot \cof \nabla_{\omega}  g (\omega)=0, \quad \forall g \in \mathcal{C}^{\infty} (\Omega,\mr^2).
\]
For this fix some $g \in \stochsmooth(\Omega,\mr^2)$ and $\varphi \in C^\infty_0(S)$. Then by definition of two-scale convergence we have
\begin{align*}	
  &\int_S \int_{\Omega} f(\omega, x)\cdot \cof \nabla_{\omega}  g(\omega) \varphi(x) \drandommeasure \td x \\
&= \lim_{\eps \downarrow 0} \int_S  \paren[\big]{\nabla^2 u^\eps(x) - \nabla^2 u^0(x)}\cdot (\cof \nabla_{\omega}  g)(T_{\eps^{-1}x} \omega) \varphi(x)  \td x \\
&= \lim_{\eps \downarrow 0} \int_S  \cof\paren[\big]{\nabla^2 u^\eps(x) - \nabla^2 u^0(x)}\cdot  \eps\nabla  \paren[\Big]{g(T_{\eps^{-1}x} \omega) \varphi(x)}  \td x \\
&\quad- \lim_{\eps \downarrow 0} \eps\int_S  \cof\paren[\big]{\nabla^2 u^\eps(x) - \nabla^2 u^0(x)}\cdot \brackets[\Big]{   g(T_{\eps^{-1}x} \omega)\otimes \nabla\varphi(x)  }\td x.
\end{align*}

The first term vanishes identically, since $\div \cof \nabla v = 0$ distributionally for all $v \in W^{1,2}(S, \mr^2)$, while the second one vanishes by the uniform bound on the integral. Since $\varphi$ was arbitrary, the claim follows.
\end{proof}
In the periodic case the purely oscillatory two-scale convergence turns out to be a good concept (see e.g., \cite{Horneuvel12}). 
The test functions considered there were fast oscillating periodic functions with vanishing mean value.
Since this implies in the periodic case a predictable rate of convergence, strong results have been obtained. We have to rely on Birkhoff's Ergodic Theorem (Theorem~\ref{thmergodic}),
which cannot provide such information. Instead we focus on derivatives of test functions, which naturally have vanishing mean
value. The following lemma states that we do not lose information by restricting ourselves to this smaller class of functions.
\begin{lemma}\label{lem:divdense}The set 
$ \set*{  \div_{\omega} v: \; v \in \stochsmooth(\Omega, \mr^n)}$ is dense in
\[
  \set*{ b \in L^2(\Omega) \Big| \int_\Omega b(\omega) \drandommeasure = 0 },
\]
with respect to the strong $L^2(\Omega)$ topology.
\end{lemma}
\begin{proof}
See \cite{zhikov1}[Lemma~2.5].
\end{proof}


The following lemma is needed for proving Lemma~\ref{lem:(31)}, which is essential for proving Theorem~\ref{theoremlowbou}.

\begin{lemma}\label{lem:product}
Let $(f^\eps) \subset W^{1,2}(S)$, $(g^\eps) \subset W^{1,2}(S)$ uniformly bounded in these spaces, and converging weakly in $W^{1,2}(S)$ to $f$ respectively $g$. Assume further that 
\[
  \limsup_{\eps\downarrow 0}\frac 1 \eps \norm { f^\eps g^\eps}_{L^1} < \infty,
\]
and that there exist  $\phi^f, \phi^g \in L^2(S, \stochsobolev^{1,2}(\Omega))$  with
\[
   \nabla' f^\eps \wtwoscale \nabla' f^0 + \nabla_\omega \phi^f, \qquad
  \nabla' g^\eps \wtwoscale \nabla' g^0 + \nabla_\omega \phi^g.
\]
Then for every $v \in \stochsmooth(\Omega,\mr^2)$ and $\varphi \in C^\infty_0(S)$ we have

\begin{eqnarray} \label{borism1}
& &\int_S \frac {f^\eps g^\eps (x)}{\eps} (\div_{\omega} v)(T_{\eps^{-1}x} \omega) \varphi(x) \td x \to \\ \nonumber & & \hspace{+10ex}\int_{\Omega \times S}  \paren[\Big]{ \phi^f(\omega, x) \cdot g(x) + f(x) \cdot  \phi^g(\omega, x)}\div_{\omega}  v(\omega) \varphi(x) \td x \drandommeasure.  
\end{eqnarray}
\end{lemma}
\begin{proof}
The proof consists of doing integration by parts 
\begin{align*}	
&\int_S \frac {f^\eps g^\eps (x)}{\eps} (\div_{\omega} v)(T_{\eps^{-1}x} \omega) \varphi(x) \td x \\
& = -\int_S  \nabla' (f^\eps g^\eps)(x)  \cdot v(T_{\eps^{-1}x} \omega) \varphi(x) \td x-\int_S (f^\eps g^\eps)(x) \nabla'\varphi(x) \cdot v(T_{\eps^{-1}x} \omega) \td x \\
& = -\int_S  \paren[\Big]{\nabla' f^\eps(x) \cdot g^\eps(x) + f^\eps(x) \cdot \nabla' g^\eps(x)}  \cdot v(T_{\eps^{-1}x} \omega) \varphi(x) \td x \\
& \hspace{+2ex}-\eps\int_S \frac{(f^\eps g^\eps)}\eps(x) \nabla'\varphi(x) \cdot v(T_{\eps^{-1}x} \omega) \td x  \\
& \quad \to 
 -\int_S\int_\Omega  \paren[\Big]{\nabla_\omega \phi^f(\omega, x) \cdot g(x) + f(x) \cdot \nabla_\omega \phi^g(\omega, x)}\cdot  v(\omega) \varphi(x) \td x \drandommeasure.
\end{align*}  
The claim now follows after integrating by parts once more, this time in $\omega$.
\end{proof}
\begin{remark}\label{borism2}
The right-hand side in \eqref{borism1} actually makes sense only by doing integration by parts since we do not have that $\phi^f(x,\cdot) g(x) + f(x)  \phi^g(x,\cdot) \in L^2(\Omega)$, for a.e. $x \in S$. However, if we would additionally know that there exists $h \in L^2( \Omega \times S)$ such that such that  for all  $v \in \stochsmooth(\Omega,\mr^n)$  and  $\varphi \in C^\infty_0(S)$ we have 
  \[
    -\int_{\Omega \times S}  \left( \nabla_{\omega}\phi^f \cdot g + f \cdot  \nabla_{\omega}\phi^g\right) \cdot v   \varphi\, \td x \,\drandommeasure= \int_{\Omega \times S} h (\div_{\omega} v) \varphi\, \td x\, \drandommeasure,
    \]
then we would be able to conclude, by the closedness property of the operator $\nabla_{\omega}$, that $\phi^f \cdot g + f \cdot \phi^g \in L^2 (\Omega \times S)$. This will be used in the proof of Lemma~\ref{lem:(31)}. 
\end{remark}



We now introduce the `mixed' spaces. The integral of $L^2(\Omega)$-valued functions will be in the sense of Bochner. 

For $A \subset \mr^n$ measurable we denote by $L^2(\Omega \times A)$ the space of measurable functions $f: \Omega \times A \to \mathbf{R}$ that satisfy $\int_A \norm f^2_{L^2(\Omega)} \td x < \infty$. We can also define the space $W^{1,2}(A, L^2 (\Omega))$ in the usual way. 

For the main part of the paper we only need $A = I$, the one-dimensional interval $\brackets*{-\frac 1 2, \frac 1 2}$. In the appendix we will however make use of this more general notion.

In the case $A = I$ we denote by $D_{x_3}$ the derivative in the $x_3$ direction, looked as the operator that maps  
$W^{1,2}(I, L^2 (\Omega))$ to $L^2( \Omega \times I)$. 
We define the space $W^{1,2}(\Omega \times I)$ as the space 
\[
   W^{1,2}( \Omega \times I)=W^{1,2}(I, L^2 (\Omega)) \cap L^2(I,W^{1,2} (\Omega)). 
\]
On the space $W^{1,2}(\Omega \times I)$ we again define the seminorm $\| \cdot\|_{\#,2}$ in the following way
  \[
    \norm u^2_{\#,2}=\norm{D_1 u}^2_{L^2(\Omega \times I)}+\norm{D_2 u}^2_{L^2(\Omega \times I)}+\norm{D_{x_3} u}^2_{L^2(\Omega \times I)}. 
      \]
By $\mathcal{W}^{1,2} ( \Omega \times I)$ we denote the completion of the space $W^{1,2}( \Omega \times I)$ with respect to the seminorm  $\| \cdot\|_{\#,2}$. 
By a density argument it can also be seen as the completion of the space $ \mathcal{C}^{\infty} (\Omega) \otimes C^{\infty} (I)$ with the same norm. We can also naturally define the operators $\nabla$ and $\div$ on these spaces. 
For $\gamma>0$ we also define $\stochsobolev_{\sym,\gamma}^{1,2}(\Omega \times I, \mr^3)$ as the completion of the space $\mathcal{C}^{\infty} (\Omega,\mathbf{R}^3) \otimes C^{\infty} (I,\mathbf{R}^3)$ with respect to the seminorm $\norm \cdot_{\#,\sym,\gamma,2}$ given by
$$
 \norm b_{\#,\sym,\gamma,2}= \|\sym (D_1 b, D_2 b, \tfrac{1}{\gamma} D_{x_3} b) \|_{L^2}, \quad \forall b \in \stochsmooth(\Omega, \mr^3) \otimes C^{\infty} (I, \mr^3).
$$
The following lemma is useful for proving Lemma~\ref{tvrdnja2}.  
\begin{lemma} \label{pomocna1} 
Let $\gamma>0$ and $f \in L^2 (  \Omega \times I, \mr^3)$ be such that 
\begin{align*}
\int_{\Omega \times I} f\cdot g=0, \quad \forall g \in  \mathcal{C}^{\infty} (\Omega,\mr^3) \otimes C^{\infty}_0(I,\mr^3)  \textrm{ that satisfy }\quad\\   D_1 g_1+D_2 g_2+\tfrac{1}{\gamma}D_{x_3} g_3 =0. 
\end{align*}
Then there exists $\psi \in \mathcal{W}^{1,2} ( \Omega \times I)$ such that
\[
   f=(D_1 \psi,  D_2 \psi, \tfrac{1}{\gamma}D_{x_3} \psi). 
\]
\end{lemma}
\begin{proof}
This follows immediately from the decomposition and density result in Theorem~\ref{thm:weyl}.
\end{proof}

We will now assume $\eps = \eps(h)$ depends additionally on $h > 0$ and satisfies $\eps(h) \downarrow 0$ if $h \downarrow 0$.
The definition of two-scale convergence extends naturally to sequences $(v^h)_{h > 0}$. We assume further that
\begin{equation}\begin{aligned}\label{def:gamma}
\gamma := \lim_{h\downarrow 0}\frac h {\eps(h)} \in (0,\infty)
\end{aligned}\end{equation}
is well-defined. In the sequel we will often suppress the dependence of $\eps(h)$ on $h$. 

Similar to Lemma~\ref{prva} implying Lemma~\ref{lem:twoscalegradient}, we can prove the following lemma, using Lemma~\ref{pomocna1}. 

\begin{lemma} \label{tvrdnja2}
Let $\gamma > 0$ be given by \eqref{def:gamma} and let $S \subset \mathbf{R}^2$ a bounded domain.  Let $(u^h)$ be a bounded sequence in $ L^2(S\times I) $, such that the sequence of scaled gradients $(\nabla_{h} u^h)$ is bounded in $L^2(S \times I,\mathbf{R}^{3})$. Assume further there exists $u^0 \in W^{1,2}(S \times I)$ such that $u^h \to u^0$ strongly in $L^2(S \times I)$. Then there exists a subsequence $h_k \to 0$, and $u^1 \in L^2(S, \mathcal{W}^{1,2}(\Omega \times I))$  such that 
\[
   \nabla_{h_k } u^{h_k} \wtwoscale (\nabla' u^0, 0) + \paren[\big]{D_1 u^1, D_2 u^1, \tfrac{1}{\gamma}D_{x_3}u^1} . 
\]
\end{lemma} 
\begin{proof} The proof  relies on the previous lemma and goes in the same way as in periodic case (see \cite{Neukamm-10} for details).
\end{proof}

\noindent The following Lemma \ref{lowersemicontinuity} shows that convex functionals are compatible with the concept of 
stochastic two-scale convergence. In the stochastic setting we can not rely on the unfolding operator (see e.g., \cite{Vis1} for the periodic case) and thus we require more to obtain the continuity of integral functionals with respect to strong stochastic two-scale convergence (see Remark \ref{borism8}). Before stating and proving the lemma we give the following definition:

\begin{definition}\label{def:stat}
Consider a measurable map $Q : \Omega  \times \mathbf{R}^n \times \mathbf{R}^{m}  \to [0,+\infty]$. 
We say that $Q$ is  $T$-stationary if for a.e.  $(\omega,x,y,v) \in \Omega \times \mathbf{R}^n \times \mathbf{R}^n \times  \mathbf{R}^{m} $  we have
\begin{equation*} \label{energyergodic} 
Q(T_{y}\omega,x, v) = Q(\omega, x+y, v).
\end{equation*} 	
\end{definition}
By $Q^0:\Omega \times \mathbf{R}^{m} \to [0,\infty)$ we denote the mapping 
$Q^0(\omega,v)=Q(\omega,0,v)$. Without loss of generality we can assume that for a.e. $x \in \mathbf{R}^n$, for all $v \in \mathbf{R}^{m}$ we have  $Q(\omega,x,v)=Q^0(T_{x} \omega, v) $. 
\begin{lemma}\label{lowersemicontinuity} 
Let $(u^{\eps})$ be a bounded sequence in $L^2(S, \mathbf{R}^m)$, such that 
\ $u^\eps \wtwoscale u^0 \in L^2(\randomspace \times S, \mathbf{R}^m)$.
 Let $Q : \randomspace \times \mr^n \times \mathbf{R}^m \to [0, \infty)$ be a $T$-stationary map such that $Q(\randomelement, x, \cdot)$ is a convex function for
 a.e.~$(\randomelement,x) \in \randomspace \times \mr^n$. Assume additionally that there exists a constant $C>0$ such that
 $Q(\omega,x,v)  \leq C(1+|v|^2)$, for a.e. $(\omega,x) \in \Omega \times \mr^n$, for all $v \in \mr^m$.  
Then for a.e. $\omega \in \Omega$ we have
\begin{eqnarray*}
 \liminf_{\eps\downarrow0} \int_S Q \paren*{  \randomelement, x/\eps,u^\eps(x) } \td x &=&  \liminf_{\eps\downarrow0} \int_S Q^0 \paren*{ T_{\eps^{-1} x} \randomelement, u^\eps(x) } \td x\\
  &\geq& \int_S \int_\randomspace Q^0\paren[\big]{ \randomelement, u^0(\randomelement, x) } \drandommeasure\td x.
\end{eqnarray*}
If additionally $u^\eps \stwoscale u^0$ then 
\begin{eqnarray*} 
 \lim_{\eps\downarrow0} \int_S Q \paren*{  \randomelement, x/\eps,u^\eps(x) } \td x &=&  \lim_{\eps\downarrow0} \int_S Q^0 \paren*{ T_{\eps^{-1} x} \randomelement, u^{\varepsilon}(x) } \td x\\
  &=& \int_S \int_\randomspace Q^0\paren[\big]{ \randomelement, u^0(\randomelement, x) } \drandommeasure\td x
\end{eqnarray*}
holds for almost every $\omega \in \Omega$.
\end{lemma}
\begin{proof}
We start with the lower semicontinuity:
Let $(u^\eps) \subset L^2(S, \mr^m)$ be uniformly bounded with $u^\eps \wtwoscale u^0$, where $u^0 \in L^2(\Omega \times S, \mr^m)$.
Then take a subsequence such that
\[
  \liminf_{\eps\downarrow0} \int_S Q^0 \paren*{ T_{\eps^{-1} x} \randomelement, u^\eps(x) } \td x
  = \lim_{k \to \infty} \int_S Q^0 \paren*{ T_{\eps_k^{-1} x} \randomelement, u^{\eps_k}(x) } \td x.
\]
Denote these limits by $M \in [0,\infty]$. If $M = \infty$, then there is nothing to show. Else
\[
  \paren*{Q^0 \paren[\big]{ T_{\eps_k^{-1} x} \randomelement, u^{\eps_k}(x) } }_{k \in \mathbb N} \subset L^1(S),
\]
with a uniform bound. Thus we may extract another subsequence (not relabeled) such that the sequence converges weakly-$*$ in measure to some $\mu$.
By the lower semicontiunity we have $\mu(S) \leq M$. We will show that
\[
  \int_\Omega   {Q^0 \paren*{ \randomelement, u^{0}(\omega, x) } } \drandommeasure \leq \frac{  \td \mu}{\td \mathcal L^n}(x),
\]
for almost every $x \in S$, where the right-hand side represents the Radon-Nikodym derivative in $x$, i.e.,
\[
\frac{  \td \mu}{\td \mathcal L^n}(x) := \lim_{r\downarrow 0} \frac{  \td \mu(B_r(x))}{\td \mathcal L^n(B_r(x))}.
\]
Let $x^*$ be a Lebesgue point of  $x \mapsto \int_{\Omega} u^0(\omega, x)\drandommeasure$ such that the limit
\[
  h(x^*) = \lim_{r\downarrow 0} \frac{  \td \mu(B_r(x^*))}{\td \mathcal L^n(B_r(x^*))}
\]
exists and such that $u^{0}(\cdot,x^*) \in L^2(\Omega,\mr^m)$. Let $a \colon \Omega \to \mr^m, b \colon \Omega \to \mr$ be measurable, bounded functions with
\begin{equation}\begin{aligned}\label{eq:affineapprox}
a(\omega)\cdot v + b(\omega) \leq Q^0(\omega, v), \quad \text{ for all $v \in \mr^m$ and almost every } \omega \in \Omega.
\end{aligned}\end{equation}
For a.e. $r > 0$ we have 
\begin{align*}	
\mu(B_r(x^*)) &= \lim_{k \to \infty} \int_{B_r(x^*)} Q^0 ( T_{\eps_k^{-1} x} \randomelement, u^{\eps_k}(x) ) \td x \\
&\geq \lim_{k \to \infty} \int_{B_r(x^*)}\left( a( T_{\eps_k^{-1} x} \randomelement )\cdot u^{\eps_k}(x) + b( T_{\eps_k^{-1} x} \randomelement) \right) \td x \\
&=  \int_{B_r(x^*)} \int_{\Omega} \left( a\paren*{ \randomelement}\cdot u^{0}(\omega,x) + b\paren*{ \randomelement} \right)\drandommeasure\td x,
\end{align*}
where we used that  $u^\eps \wtwoscale u^0$.
Therefore
\begin{align*}	
h(x^*) &\geq \lim_{r \downarrow 0}  \frac 1 { \abs {  B_r(x^*)}}\int_{B_r(x^*)} \int_{\Omega}\left( a\paren*{ \randomelement}\cdot u^{0}(\omega,x) + b\paren*{ \randomelement}\right) \drandommeasure\td x \\
&=  \int_{\Omega}\left( a\paren*{ \randomelement}\cdot u^{0}(\omega,x^*) + b\paren*{ \randomelement}\right) \drandommeasure.
\end{align*}
By taking the supremum over the functions $a,b$ satisfying \eqref{eq:affineapprox} we obtain
\[
 h(x^*) \geq\int_{\Omega} Q^0\paren*{ \randomelement,u^{0}(\omega,x^*)} \drandommeasure.
\]
Integrating both sides w.r.t.\ $x^*$ yields the first claim.

For the continuity assume that $u^\eps \stwoscale u^0$ and assume that $u^0 \in L^2(\Omega, \mr^m) \otimes L^2(S, \mr^m)$. Then from the strong convergence follows 
\[ \| u^{\eps} (x)-u^0(T_{\varepsilon^{-1}x}\omega,x)\|_{L^2(S)} \to 0, \textrm{ for a.e. } \omega \in \Omega. \]
From the convexity and the uniform bound of $Q$ it follows that there exists a constant $C>0$ such that
\begin{equation} \label{borism5}
|Q^0(\omega,v_1)-Q^0(\omega,v_2)| \leq C(1+|v_1|+|v_2|) |v_1-v_2|, \forall v_1,v_2 \in \mr^n.
\end{equation} 
Using \eqref{borism5} and the Ergodic theorem we conclude 
\begin{align*}	
 \lim_{\varepsilon \downarrow 0} \int_S Q^0( T_{\eps^{-1}x}\omega, u^\eps(x)) \td x &= \lim_{\varepsilon \downarrow 0}  \int_S Q^0( T_{\eps^{-1}x}\omega, u^0(T_{\eps^{-1}x'}\omega,x))  \td x\\
  &= \int_\Omega\int_S Q^0(\omega, u^0 (\omega,x)) \td x\drandommeasure.
\end{align*}
For general $u^0 \in L^2(\Omega \times S, \mr^m)$ the claim follows by approximation and using \eqref{borism5}. 
\end{proof}
\begin{remark} \label{borism8} 
	Notice that in the proof of the second claim we only use the relation \eqref{borism5}, but not the convexity. The reason why we treated separately the case $u^0 \in L^2(\Omega, \mr^m) \otimes L^2(S, \mr^m)$ is because for $u^0 \in L^2(\Omega \times S)$ one can not guarantee that $u^0 (T_{\varepsilon^{-1}x} \omega,x)$ is measurable (see \cite{mikelic} for details).
\end{remark}
\begin{remark}\label{borism6} 
Lemma~\ref{lowersemicontinuity} also holds for bounded sequence in $L^2(S \times I, \mr^m)$ which stochastically two-scale converges in the sense of Remark \ref{borism7}.
\end{remark}

\section{Homogenization of the plate model}
\subsection{General framework and main result}
In this chapter $S \subset \mathbf{R}^2$ is a bounded domain and the interval $I=[-\tfrac{1}{2},\tfrac{1}{2}]$. 
Let $\gamma$ be as in \eqref{def:gamma}.
The main results are Theorem~\ref{theoremlowbou} (lower bound) and Theorem \ref{theoremupperbound} (upper bound). 
To prove the $\Gamma$-limit result we will need some additional assumption on the domain $S$. We will assume that the domain $S$ is piecewise $C^1$. This assumption is necessary only for the proof of upper bound, and can be weakened (see Theorem~\ref{theoremupperbound} for a precise definition). For the lower bound we only require $S$ to be a Lipschitz domain.

Consider a measurable map $W : \Omega  \times \mathbf{R}^2 \times \mathbf{R}^{3 \times 3}  \to [0,+\infty]$, representing the stored energy density function,
satisfying the following:
\begin{assumption}\label{assumption} 
We assume that $W$ is $T$-stationary as in Definition~\ref{def:stat} and that
  $W(\omega, x',\cdot)$ is continuous  on $\mathbf{R}^{3\times 3}$ for a.e. $(\omega,x') \in \Omega \times  S$. 
  This will ensure the measurability of all composition mappings that appear (see, e.g., the expression \eqref{correction1}) 
We also assume that the following properties are satisfied:
\begin{enumerate} 
\item Objectivity property
\begin{equation*}\begin{aligned}\label{ass:frame-indifference}
 & W(\omega,x', RF)=W(\omega,x', F) \\  &\hspace{+5ex}  \textrm{ for
  a.e. $(\omega,x') \in \Omega \times \mathbf{R}^{2}$, for    all $F\in\mathbf{R}^{3\times 3}$, $R\in\SO 3$.}
  \end{aligned}\end{equation*}
\item  There exist constants $c_1,c_2, \rho >0$ such that 
\begin{equation}
\label{ass:non-degenerate}
\begin{split}
W(\cdot, \cdot,  F) &\geq c_1\dist^2( F,\SO 3),\text{ a.e.\ on $\Omega \times S$ and for all
	$ F\in\mathbf{R}^{3\times 3}$}\\
W(\cdot,\cdot , F) &\leq c_2\dist^2( F,\SO 3),\text{ a.e.\ on $\Omega \times S$ and for all
	$ F\in\mathbf{R}^{3\times 3}$} \\ & \hspace{+35ex}\text{  with $\dist^2( F,\SO 3)\leq\rho$. }
\end{split}
\end{equation}
\item  There exists a monotone function $r:[ 0,\infty) \to [0,\infty]$ with $r(t) \downarrow 0$ as $t \downarrow 0$ such that, for a.e. $(\omega,x')  \in \Omega \times \mathbf{R}^2$, there exists a quadratic form $Q(\omega,x',\cdot)$ on $\mathbf{R}^{3 \times 3}$  with  
\begin{equation} \label{Taylorexpansion} 
\abs{W(\omega,x', I +G)-Q(\omega,x',G)}\leq r(\abs G)\abs G^2 \textrm{ for all } G \in  \mathbf{R}^{3×3}.
\end{equation} 
\end{enumerate} 
\end{assumption}
For $\omega \in \Omega$ we define the energy functionals $I^h:  W^{1,2}(S \times I, \mathbf{R}^{3 })\to [0,\infty] $ by
\begin{equation} \label{correction1}
  I^{h}(\de)= \frac{1}{h^2} \int_{S\times I} W\paren[\big]{\omega,x'/\varepsilon, \nabla_h \de(x', x_3)} \td x'\td x_3.
\end{equation} 
As a consequence of relations  \eqref{ass:non-degenerate}-\eqref{Taylorexpansion} we have the following lemma.

\begin{lemma}
	\label{lem:1}
	Let $W$ be as in Assumption~\ref{assumption} and let $Q$ be the quadratic form
	associated with $W$ via \eqref{Taylorexpansion}. Then
	\begin{enumerate}[(i)]
		\item[(Q1)] $Q$ is $T$-stationary,
		\item[(Q2)] for  a.e. $(\omega,x') \in \Omega \times \mathbf{R}^2$ we have that 
		\begin{equation*}\label{Qelliptisch}
		c_1\abs{\sym G}^2\leq Q(\omega,x',G)=Q(\omega,x',\sym G)\leq c_2\abs{\sym G}^2,\ \forall G\in\mathbf{R}^{3\times 3}.
		\end{equation*}
	\end{enumerate}
\end{lemma}

As before by $Q^0:\Omega \times \mathbf{R}^{3 \times 3} \to [0,\infty)$ we denote the mapping 
$Q^0(\omega,G)=Q(\omega,0,G)$. Again without loss of generality we can assume that for a.e. $x' \in \mathbf{R}^2$, for all $G \in \mathbf{R}^{3 \times 3}$ we have  $Q(\omega,x',G)=Q^0(T_{x'} \omega, G) $.

\begin{definition}[The relaxation formula]\label{def:Qgamma}
Let $\gamma >0$ and define the map $\mathcal{Q}^{\gamma}:\mathbf{R}^{2 \times 2} \to [0,\infty)$ as follows
\begin{equation} \label{defQ}
\mathcal{Q}^{\gamma} (G)= \inf_{\phi, B } \int_{\Omega \times I} Q^0\paren[\big]{\omega,\iota(B+x_3 G)+\sym(D_1 \phi, D_2 \phi, \tfrac{1}{\gamma} D_{x_3} \phi)}  \drandommeasure \td x_3,
\end{equation}
where the infimum is taken over $B \in \mathbf{R}^{2 \times 2}$ and $\phi \in \mathcal{W}_{\sym,\gamma}^{1,2}(\Omega \times I, \mr^3)$. 
\end{definition}
It can be shown that $\mathcal{Q}^{\gamma}$ is the quadratic form which is coercive on symmetric matrices. 
Namely the expression on the right-hand side of \eqref{defQ} can be seen as the projection of $x_3 G$ on the closed subspace of $L^2(\Omega \times I, \mr^{3 \times 3}_{\sym})$ defined by $\iota(\mr^{2 \times 2}_{\sym})\oplus \stochsobolev_{\sym,\gamma}^{1,2}(\Omega \times I, \mr^3) $ (the orthogonal decomposition) in the norm defined by the quadratic form $Q^0$. The coercivity property follows easily from the coercivity property of $Q^0$.  

In the bending regime we assume that the sequence of minimizers $(u^h)$ satisfies
$$ \limsup_{h \downarrow 0} I^h(u^h) < \infty. $$
By the compactness result (see Lemma~\ref{lemmacomp}), it can be concluded that the limit deformations are Sobolev isometries. 
By $W^{2,2}_{\textrm{iso}}(S)$ we denote the set
\begin{eqnarray*}
	W^{2,2}_{\textrm{iso}}(S)=\set{\de \in W^{2,2} (S,\mathbf{R}^3): \partial_{\alpha}\de \cdot \partial_{\beta} \de= \delta_{\alpha \beta} \; \text{ for } \alpha, \beta = 1,2},
\end{eqnarray*}
where $\delta$ denotes the Kronecker delta symbol.  For $\de \in W^{2,2}_{\textrm{iso}}(S)$ we define its normal $n^\de \in W^{1,2} (S,\mathbf{R}^3)$ as $n^\de=\partial_1 \de\wedge \partial_2 \de$ and the second fundamental form $ \secf^\de$ as
\[
 \secf^\de_{\alpha \beta}=\partial_{\alpha} \de \cdot \partial_{\beta} n^\de=-\partial_{\alpha \beta}\de \cdot n^\de, \quad \alpha,\beta=1,2.
\]
We define the limit functional $I^0:  W^{2,2}_{\textrm{iso}}(S) \to [0, \infty)$ in the following way
\begin{align*}	
  I_\gamma^0(\de) =\int_S \mathcal{Q}^{\gamma}(\secf^\de ( x' )) \td x'. 
\end{align*}
The following compactness result is the consequence of the compactness result given in \cite{FJM-02} and is explained in \cite{Vel13}[Lemma~3.3, Remark~4].

\begin{lemma}\label{lemmacomp}
There exist constants $C,c > 0$, depending only on $S$, such that for every $u \in W^{1,2}(S \times I, \mr^3)$ there exists: a map $R : S \to \SO3$, which is piecewise constants on cubes $x' + h[0,1)^2, \, x' \in h\mz^2$, as well as $\widetilde R \in W^{1,2}(S, \mr^{3\times3})$ such that for every $\xi \in \mr^2$ with $\abs \xi_\infty = \max\set{ \abs{\xi \cdot e_1}, \abs{ \xi \cdot e_2}} \leq h$ and for each $S' \subset S$ with $\dist (S', \partial S) > ch$ we have
\begin{align*}	
\norm { \nabla_h u - R}_{L^2(S' \times I)}^2 + \norm{ R - \widetilde R}_{L^2(S')}^2 + h^2 \norm { R - \widetilde R}_{L^\infty(S')}^2 \phantom{\hspace{3.5cm}}\\
 + h^2 \norm{ \nabla' \widetilde R}_{L^2(S')}^2 +\norm{ R( \cdot + \xi) - R}_{L^2(S')}^2 \leq C \norm { \dist(\nabla_h u, \SO3)}_{L^2(S \times I)}^2.
    \end{align*}
If additionally $S'$ is open with $\partial S'$ of class $C^{1,1}$, then there exists $\widetilde u \in W^{2,2}(S')$ such that
\begin{align*}	
h^2 \norm { \widetilde u}_{W^{2,2}(S')}^2 + \norm { \nabla' \widetilde u - (\widetilde R e_1, \widetilde R e_2)}_{L^2(S')}^2 + \norm { \nabla' \widetilde u - \nabla' \overline u}_{L^2(S')}^2 \phantom{\hspace{2cm}}\\
 \leq C \norm { \dist(\nabla_h u, \SO3)}_{L^2(S \times I)}^2,
\end{align*}
where $\overline u = \int_I u(x_3) \td x_3$.
\end{lemma}
\begin{remark}\label{rem:constructions}
The existence of the function $R$ follows from the geometric rigidity, proved in \cite{FJM-02},
while $\widetilde R$ is the mollification of $R$ on scale $h$. The function $\widetilde u$ 
is the projection of $\widetilde R$ onto gradient fields.
\end{remark}
The following two theorems are the main result of this paper. They correspond to the statement of lower and upper bound for the $\Gamma$-limit. 


\begin{theorem}\label{theoremlowbou} Let $S \subset \mathbf{R}^2$ be a bounded domain with Lipschitz boundary. 
        Let $ (u^{h}) \subset W^{1,2}(S \times I, \mathbf{R}^3)$ be a family with finite elastic energy, i.e.\
          \[
              \limsup_{ h \downarrow 0 } I^h( \de^h) < \infty.
          \]
        \begin{enumerate}
        \item There exists $u \in  W^{2,2}_{\textrm{iso}} (S)$ such that (up to a subsequence) we have
        \begin{align}	
        \label{add1} \de^{h}- \int_S \de^{h} &\to \de        & \textrm{ strongly in } W^{1,2}(S \times I, \mr^3), \\
        \label{add3} \nabla_{h} \de^{h} &\to (\nabla' \de, n^{u})&                   \textrm{ strongly in } L^2(S \times I, \mr^{3\times 3}).
        \end{align}
                \item 
                      For a.e. $\omega \in \Omega$ and any sequence $(u^h)$ satisfying \eqref{add1}, \eqref{add3} for some $u\in W^{2,2}_{\textrm{iso}}(S)$ we have that
                 \[
                  \liminf_{h \to 0}   I^{h} (\de^{h}) \geq I^0_\gamma(u).
                 \]
        \end{enumerate} 
\end{theorem}
\begin{remark}\label{borism4}
	The claim a is the standard compactness result for the bending regime, whose proof can be found in \cite{FJM-02}. 
\end{remark}
\begin{theorem}\label{theoremupperbound}
Let $S \subset \mathbf{R}^2$ be a bounded domain with Lipschitz boundary, such that its normal is continuous away from a subset of $\partial S$ with length zero (e.g., the boundary is piecewise $C^1$). 
	Let $\de \in W^{2,2}_{\textrm{iso}}(S)$. 
Then for a.e.\ every $\omega \in \Omega$  there exists a sequence $(u^h) \subset  W^{1,2} (S \times I, \mathbf{R}^3)$ such that we have
	\begin{enumerate} 
		\item $\de^h \to \de$ strongly in $W^{1,2}(S \times I, \mr^3)$;
		\item $I^h ( \de^h ) \to I^0_\gamma (u)$.
	\end{enumerate} 	
\end{theorem}

\subsection{Identifications of two-scale limits and proof of Theorem~\ref{theoremlowbou}}
\subsubsection{Two-scale limits of the most important terms}
In this section we explicitly compute the two-scale limits, which will be needed to prove the lower bound stated by Theorem~\ref{theoremlowbou}.
\begin{lemma}\label{lem:(31)}
Let $S'\subset \mr^2$ be a bounded domain. Let $(\widetilde u^h)\subset W^{2,2}(S')$, $(R^h) \subset L^\infty(S', \SO3)$ and $(\widetilde R^h )\subset W^{1,2}(S', \mr^{3\times3})$ with
\begin{align*}	
h^2 \norm { \widetilde u^h}_{W^{2,2}(S')}^2+ \norm { \nabla' \widetilde u^h - (\widetilde R^h e_1, \widetilde R^h e_2)}_{L^2(S')}^2 +\norm{ R^h - \widetilde R^h}_{L^2(S')}^2\phantom{\hspace{2cm}}\\
 + h^2 \norm { R^h - \widetilde R^h}_{L^\infty(S')}^2  + h^2 \norm{ \nabla' \widetilde R^h}_{L^2(S')}^2 \leq C h^2.
    \end{align*}
Then there exist a (not relabeled) subsequence and functions $w^0_\alpha \in L^2(S')$, $\phi^{ \widetilde u} \in L^2(S', \stochsobolev^{2,2}(\Omega, \mr^3))$, $\phi^{\widetilde R} \in L^2(S' ,  \stochsobolev^{1,2}(\Omega, \mr^3))$ such that for $\alpha=1,2$ we have
\begin{eqnarray*}   
  \frac {\scalar{ R^h e_\alpha, \widetilde R^h e_3 } + \scalar {R^h e_3, \partial_\alpha \widetilde u^h}}h &\wtwoscale& \frac 1 \gamma w^0_\alpha + \frac 1 \gamma \scalar* {R e_3, D_{\alpha}\phi^{\widetilde u} } + \frac 1 \gamma \scalar { \phi^{\widetilde Re_3}, R e_\alpha },\\
  \nabla^2 \widetilde{u}^h &\wctwoscale& \nabla^2_{\omega} \phi^{\widetilde{u}}, \\
  \nabla (\widetilde{R}^he_3) & \wctwoscale& \nabla_{\omega}\phi^{\widetilde{R}e_3},
\end{eqnarray*}
and $\scalar{R e_3, D_{\alpha}\phi^{ \widetilde u} } + \scalar { \phi^{\widetilde Re_3}, R e_\alpha } \in L^2(S'\times \Omega ,\mr^3)$. 
\end{lemma}
\begin{proof}
Notice that
\[
 \scalar{ R^h e_\alpha, \widetilde R^h e_3 } + \scalar {R^h e_3, \partial_\alpha \widetilde u^h}  = \scalar{ R^h e_3 - \widetilde R^h e_3, \partial_\alpha\widetilde u^h - R^h e_\alpha} + \scalar {  \widetilde R^h e_3, \partial_\alpha \widetilde u^h}.
\]
The left-hand side is of order $h$, while the first term on the right-hand side is of order $h^2$. Thus the second term is of order $h$. After dividing by $h$ the first term on the right-hand side converges strongly to $0$ as $h\to 0$ and thus does not contribute to the two-scale limit.
We define 
\[
  f^h:=  \widetilde R^h e_3, \quad g^h_{\alpha} := \partial_\alpha \widetilde u^h \quad \alpha = 1,2,
\]
and notice that after extracting a subsequence the components $f^h_i,{(g^h_\alpha)}_i$, $i = 1,2,3$ satisfy the assumptions of Lemma~\ref{lem:product} (see also Remark \ref{borism2}).
Thus
\[
  \frac 1 h \scalar*{f^h, g^h_\alpha} = \frac \eps h \paren*{\frac 1 \eps \scalar*{f^h, g_\alpha^h}} \wtwoscale \frac 1 \gamma w^0_\alpha + \frac 1 \gamma \scalar* {f, \phi^{g_\alpha} } + \frac 1 \gamma \scalar { \phi^f, g_\alpha },
\]
for some $w^0_\alpha \in L^2(S')$ and $\phi^f, \phi^{g_\alpha} \in L^2(S' , \stochsobolev^{1,2}(\Omega,\mr^3))$ such that  $\scalar* {f, \phi^{g_\alpha} } + \scalar { \phi^f, g_\alpha } \in L^2 (S' \times \Omega, \mr^3)$. From Lemma~\ref{lem:twoscalesecondgradient} we additionally deduce that there exists $ \phi^g \in L^2(S', \stochsobolev^{2,2}(\Omega,\mr^3))$ with $D_\alpha  \phi^g = \phi^{g_\alpha}$ for $\alpha = 1,2$. This yields
\[
  \scalar{ R^h e_\alpha, \widetilde R^h e_3 } + \scalar {R^h e_3, \partial_\alpha \widetilde u^h}\wtwoscale \frac 1 \gamma w^0_\alpha + \frac 1 \gamma \scalar* {R e_3, D_{\alpha}\phi^{\widetilde u} } + \frac 1 \gamma \scalar{ \phi^{\widetilde Re_3}, Re_\alpha }
\]
for some $ w^0_\alpha  \in L^2(S'),\ \phi^{\widetilde Re_3} \in L^2(\Omega \times S', \mr^3)$ and $\phi^{ \widetilde u} \in L^2(S', \stochsobolev^{2,2}(\Omega, \mr^3))$. 
\end{proof}

The following lemma identifies the most sensitive term in our analysis.
\begin{lemma}\label{lem:(33)}
Let $S' \subset \mr^2$ be a bounded domain. Let $(\widetilde R^h) \subset W^{1,2}(S', \mr^{3\times3})$, and let $(R^h) \subset L^\infty(S', \SO3)$ be such that 
for each $h > 0$ the map $R^h$ is piecewise constant on each cube $x' + h[0,1)^2$ with $x' \in h\mz^2$. Assume further that for each $\xi \in \mr^2$ with $\abs {\xi}_\infty \leq h$ we have
\begin{align*}	
 \norm{ R^h - \widetilde R^h}_{L^2(S')}^2 + h^2 \norm { R^h - \widetilde R^h}_{L^\infty(S')}^2  + h^2 \norm{ \nabla' \widetilde R^h}_{L^2(S')}^2 \phantom{\hspace{2cm}}\\
  +\norm{ R^h( \cdot + \xi) - R^h}_{L^2(S^h)}^2 \leq C h^2
    \end{align*}
for each sequence of subdomains $ S^h \subset S'$ which satisfy $\dist( S^h, \partial  S') \geq ch$ for some $c > 0$. 
Finally assume that $\widetilde R^h$ is the mollification of $R^h$ on scale $h$.

Then there exist $w^0_3 \in L^2(S')$ and $\phi^{\widetilde R e_3} \in L^2(S', \stochsobolev^{1,2}(\Omega, \mr^3))$ such that on a subsequence
\begin{eqnarray*}	
\frac {\scalar{ R^h e_3,\widetilde R^h e_3 } - 1} h &\wtwoscale & \frac 1 \gamma w^0_3 +  \frac { 1} { \gamma} \scalar {R e_3,  \phi^{\widetilde Re_3}}, \\
\nabla (\widetilde{R}^h e_3) &\wctwoscale &  \nabla \phi^{\widetilde{R}e_3},
  \end{eqnarray*}
with $\scalar {R e_3,  \phi^{\widetilde R e_3}} \in L^2(S' \times \Omega)$. 
\end{lemma}
\begin{proof}
From
\[
  f^h := \frac {\scalar{ R^he_3, \widetilde R^he_3} -   1} h= \frac 1 h\scalar { R^h e_3, \widetilde R^he_3 - R^h e_3}
\]
we easily see that $(f^h)$ is uniformly bounded in $L^2(S')$. Thus up to a subsequence we have
\[
  f^h \weakly \frac 1 \gamma w^0_3 \quad \text{ and } f^h \wtwoscale \frac 1 \gamma w^0_3 + \phi
\]
for some $w^0_3 \in L^2(S')$ and $\phi \in L^2(\Omega \times S')$. To further identify $\phi$ we test the sequence against derivatives. For this fix some $b \in \stochsmooth(\Omega)$ and $\varphi \in C^\infty_0(S')$. Let $h > 0$ be small enough and such that there is a subdomain $S^h \subset S'$ with $\dist(S^h, \partial S') \geq ch$ and the compact support $K$ of $\varphi$ is contained in $S^h$.

First note that
\begin{align*}	
\int_{K} f^h(x') (D_{\alpha} b)(T_{\eps^{-1}x'} \omega)\varphi(x') \td x' 
&= 
\eps\int_{K} f^h(x') \partial_\alpha [b(T_{\eps^{-1}x'} \omega)\varphi(x')] \td x' \\
&\quad-\eps\int_{K} f^h(x') b(T_{\eps^{-1}x'} \tomega)\partial_\alpha\varphi(x') \td x' .
\end{align*}
The last term converges to 0, and so we focus on the first. For this we define $q_z := (z +  h[0,1)^2)\intersect {K}$ and compute

\begin{align*}	
\eps\int_{K} f^h \partial_\alpha [b(T_{\eps^{-1}x'} \tomega)\varphi] \td x' 
  &=\frac {\eps}{h} \int_{K} \scalar {R^h e_3, \widetilde R^h e_3} \partial_{\alpha}  [b(T_{\eps^{-1}x'} \omega)\varphi] \td x'\\
  &=\frac {\eps}{h}\sum_{z \in h\mz^2 } \int_{q_z} \scalar {R^h(z) e_3, \widetilde R^h e_3} \partial_{\alpha}  [b(T_{\eps^{-1}x'} \omega)\varphi] \td x'\\
  &=\frac {\eps}{h}\sum_{z \in h\mz^2 } \int_{q_z} \partial_\alpha \brackets*{\scalar {R^h(z) e_3, \widetilde R^h e_3} b(T_{\eps^{-1}x'} \omega)\varphi} \td x'\\
  &\quad -\frac {\eps}{h}\sum_{z \in h\mz^2 } \int_{q_z}  \brackets*{\scalar {R^h(z) e_3, \partial_\alpha\widetilde R^h e_3} b(T_{\eps^{-1}x'} \omega)\varphi} \td x'.
\end{align*}
For the last term we use Lemma~\ref{lem:twoscalegradient} to conclude there exists $\phi^{\widetilde Re_3} \in L^2(S', \stochsobolev^{1,2}( \Omega, \mr^3))$ with
\[
 \nabla' \widetilde R^h e_3 \wtwoscale \nabla' Re_3 + \nabla_\omega \phi^{\widetilde Re_3} . 
\]
Together with $R^h \to R $ strongly in $L^2(S')$, we obtain 
\begin{align*}	
&-\frac {\eps}{h}\sum_{z \in h\mz^2 } \int_{q_z}  \brackets*{\scalar {R^h(z) e_3, \partial_\alpha\widetilde R^h e_3} b(T_{\eps^{-1}x'} \omega)\varphi} \td x' \\
& = -\frac { \eps } h\int_{{K}}  \brackets*{\scalar {R^h e_3, \partial_\alpha\widetilde R^h e_3} b(T_{\eps^{-1}x'} \omega)\varphi} \td x'\\
&\quad \xrightarrow{\quad} \frac { -1} { \gamma} \int_\Omega \int_{K} \brackets*{\scalar {R e_3, \partial_\alpha R e_3 + D_{\alpha} \phi^{\widetilde R e_3}} b(\omega)\varphi} \td x'\drandommeasure\\
& \qquad= \frac { -1} { \gamma} \int_\Omega \int_{K} \brackets*{\scalar {R e_3, D_{\alpha} \phi^{\widetilde Re_3}} b(\omega)\varphi} \td x'\drandommeasure\\
& \qquad= \frac { 1} { \gamma} \int_\Omega \int_{K} \brackets*{\scalar {R e_3,  \phi^{\widetilde Re_3}} D_{\alpha}b(\omega)\varphi} \td x'\drandommeasure.
\end{align*}
Next we show that
\[
  \frac {\eps}{h}\sum_{z \in h\mz^2 } \int_{q_z} \partial_\alpha \brackets*{\scalar {R^h(z) e_3, \widetilde R^h e_3} b(T_{\eps^{-1}x'} \tomega)\varphi} \td x' \xrightarrow{h \to 0} 0.
\]
If $q_z \neq \emptyset$ then let $\Gamma^{\textrm{pos}}_z, \Gamma^{\textrm{neg}}_z$ be the boundary of $q_z$ perpendicular to $e_\alpha$ with normals $e_\alpha$ resp.\ $-e_\alpha$, else $\Gamma^{\textrm{pos}}_z,\Gamma^{\textrm{neg}}_z := \emptyset$. The fundamental theorem together with Fubini's Theorem yields
\begin{align*}	
 &\frac {\eps}{h}\sum_{z \in h\mz^2 } \int_{q_z} \partial_\alpha \brackets*{\scalar {R^h(z) e_3, \widetilde R^h e_3} b(T_{\eps^{-1}x'} \omega)\varphi} \td x'  \\
 &=\frac {\eps}{h}\sum_{z \in h\mz^2 } \int_{\Gamma^{\textrm{pos}}_z}  \scalar {R^h(z) e_3, \widetilde R^h e_3} b(T_{\eps^{-1}x'} \omega)\varphi \td x'\\
 &\quad-\frac {\eps}{h}\sum_{z \in h\mz^2 } \int_{\Gamma^{\textrm{neg}}_z}  \scalar {R^h(z) e_3, \widetilde R^h e_3} b(T_{\eps^{-1}x'} \omega)\varphi \td x'.
\end{align*}
We rearrange the sum and obtain
\begin{align*}	
 &\frac {\eps}{h}\sum_{z \in h\mz^2 } \int_{\Gamma^{\textrm{pos}}_z}  \scalar {R^h(z) e_3, \widetilde R^h e_3} b(T_{\eps^{-1}x'} \omega)\varphi \td x'\\
 &\quad-\frac {\eps}{h}\sum_{z \in h\mz^2 } \int_{\Gamma^{\textrm{neg}}_z}  \scalar {R^h(z) e_3, \widetilde R^h e_3} b(T_{\eps^{-1}x'} \omega)\varphi \td x' \\
 &=\frac {\eps}{h}\sum_{z \in h\mz^2 } \scalar* {R^h(z) e_3 - R^h( z + h e_\alpha)e_3, \int_{\Gamma^{\textrm{pos}}_z}  \widetilde R^h e_3 b(T_{\eps^{-1}x'} \omega)\varphi \td x'}.
\end{align*}
By assumption
\[
\frac { R^h(z)  - R^h( z + h e_\alpha) } h 
\]
is uniformly bounded in $z$ and $h$, which implies 
\[
  \limsup_{ h \downarrow 0 } \sum_{z \in h\mz^2 } \abs[\big]{ R^h(z)  - R^h( z + h e_\alpha)}^2 < \infty.
\]
Denote by $\mathcal Z \subset h\mz^2$ the $z$-values such that $\Gamma^{\textrm{pos}}_z $ has positive measure . Using the trace inequality and the Poincar\'e's inequality afterwards, we get for such $z$ that
\[
\int_{\Gamma^{\textrm{pos}}_z}  \abs [\Big] { \widetilde R^h  - \frac 1 {\abs{\Gamma^{\textrm{pos}}_z}}\int_{\Gamma^{\textrm{pos}}_z} \widetilde R^h }^2 \leq C h \int_{q_z} \abs {\nabla \widetilde R^h}^2 \td x'.
\]
Combing both previous statements we see that
\begin{align*}	
&\lim_{h\to0} \frac {\eps}{h}\sum_{z \in h\mz^2 } \scalar* {R^h(z) e_3 - R^h( z + h e_\alpha)e_3, \int_{\Gamma^{\textrm{pos}}_z}  \widetilde R^h e_3 b(T_{\eps^{-1}x'} \omega)\varphi \td x'}\\
&  =\frac 1 \gamma \lim_{h\to0} \sum_{z \in \mathcal Z } \scalar* {R^h(z) e_3 - R^h( z + h e_\alpha)e_3, \frac 1 {\abs{\Gamma^{\textrm{pos}}_z}}\int_{\Gamma^{\textrm{pos}}_z}\widetilde R^h e_3} \int_{\Gamma^{\textrm{pos}}_z}   b(T_{\eps^{-1}x'} \omega)\varphi \td x'.
\end{align*}
Noticing the uniform bound
\[
\abs* { \int_{\Gamma^{\textrm{pos}}_z}   b(T_{\eps^{-1}x'} \omega)\varphi \td x'}\leq h \norm{ b }_{L^\infty(\Omega)} \norm{ \varphi}_{L^\infty(S')}
\]
we only need to show that
\[
\limsup_{h\downarrow 0} \sum_{z \in \mathcal Z } \abs*{\scalar* {R^h(z) e_3 - R^h( z + h e_\alpha)e_3, \frac 1 {\abs{\Gamma^{\textrm{pos}}_z}}\int_{\Gamma^{\textrm{pos}}_z}\widetilde R^h e_3}} < \infty
\]
to conclude the vanishing of the product. For this bound note that $\widetilde R^h$ is the mollification of $R^h$ on scale $h$. Therefore there exist $z$-independent constants $0 \leq \eta_1, \eta_2, \eta_3 \leq 1$
with
\begin{align*}	
\frac 1 {\abs{\Gamma^{\textrm{pos}}_z}}\int_{\Gamma^{\textrm{pos}}_z}\widetilde R^h = &\eta_1 \paren[\Big]{R^h(z) + R^h(z+h e_\alpha)} \\
    &\quad +  \eta_2 \paren[\Big]{ R^h(z + h e_\alpha^\perp) + R^h(z+h (e_\alpha+e_\alpha^\perp)) }\\
    &\quad +  \eta_3 \paren[\Big]{ R^h(z - h e_\alpha^\perp) + R^h(z+h (e_\alpha-e_\alpha^\perp)) }.
    \end{align*}
We compute
\begin{align*}	
I &:= \eta_1\sum_{z \in \mathcal Z} \abs*{\scalar* {R^h(z) e_3 - R^h( z + h e_\alpha)e_3,  R^h(z)e_3 + R^h(z+h e_\alpha)e_3}} \\
&= \eta_1\sum_{z \in \mathcal Z} \brackets[\Big]{ |R^h(z) e_3|^2 - |R^h(z+h e_\alpha)e_3|^2 } = 0.
\end{align*}
With this result we easily obtain
\begin{align*}	
&II := \eta_2\sum_{z \in \mathcal Z} \abs*{\scalar* {R^h(z) e_3 - R^h( z + h e_\alpha)e_3,  R^h(z + h e_\alpha^\perp)e_3 + R^h(z+h (e_\alpha+e_\alpha^\perp))e_3}} \\
&   \leq \eta_2\sum_{z \in \mathcal Z} \abs*{\scalar* {R^h(z) e_3 - R^h( z + h e_\alpha)e_3,  R^h(z + h e_\alpha^\perp)e_3 - R^h(z )e_3   }} \\
&   + \eta_2\sum_{z \in \mathcal Z} \abs*{\scalar* {R^h(z) e_3 - R^h( z + h e_\alpha)e_3,     R^h(z+h (e_\alpha+e_\alpha^\perp))e_3 - R^h(z + he_\alpha)e_3 }} < \infty
\end{align*}
and analogously also that
\begin{align*}	
III := \eta_3\sum_{z \in \mathcal Z} \abs*{\scalar* {R^h(z) e_3 - R^h( z - h e_\alpha)e_3,  R^h(z + h e_\alpha^\perp)e_3 + R^h(z+h (e_\alpha-e_\alpha^\perp))e_3}} 
       \end{align*}
is uniformly bounded. Obviously
\begin{align*}	
\sum_{z \in \mathcal Z} \abs*{\scalar* {R^h(z) e_3 - R^h( z + h e_\alpha)e_3, \frac 1 {\abs{\Gamma^{\textrm{pos}}_z}}\int_{\Gamma^{\textrm{pos}}_z}\widetilde R^h e_3}} 
\leq I + II + III,
\end{align*}
and we conclude that (after possibly redefining $w_3^0$)
\begin{equation*}\begin{aligned}\label{eq:(3,3)component}
\frac 1 h\paren[\Big]{ \scalar{ R^h e_3,\widetilde R^h e_3 } - 1} \wtwoscale \frac 1 \gamma w^0_3 +  \frac { 1} { \gamma} \scalar {R e_3,  \phi^{\widetilde Re_3}}. 
  \end{aligned}\qedhere\end{equation*}
\end{proof}

\begin{lemma}\label{lem:cofortho}
Let $S'$ be a bounded Lipschitz domain and let $(\widetilde u^h)\subset W^{2,2}(S', \mr^3)$, $(\widetilde R^h) \subset W^{1,2}(S', \mr^{3\times3})$ and $(R^h) \subset L^\infty(S', \SO3)$ be such that 
for each $h > 0$ the map $R^h$ is piecewise constant on each cube $x' + h[0,1)^2$ with $x' \in h\mz^2$, and for each $\xi \in \mr^2$ with $\abs {\xi}_\infty \leq h$ we have
\begin{align*}	
h^2 \norm { \widetilde u^h}_{W^{2,2}(S')}^2+ \norm { \nabla' \widetilde u^h - (\widetilde R^h e_1, \widetilde R^h e_2)}_{L^2(S')}^2 +\norm{ R^h - \widetilde R^h}_{L^2(S')}^2\phantom{\hspace{2cm}}\\
 + h^2 \norm { R^h - \widetilde R^h}_{L^\infty(S')}^2  + h^2 \norm{ \nabla' \widetilde R^h}_{L^2(S')}^2 +\norm{ R^h( \cdot + \xi) - R^h}_{L^2(S^h)}^2 \leq C h^2,
    \end{align*}
for some $C > 0$ and for each sequence of subdomains $ S^h \subset S'$ which satisfy $\dist( S^h, \partial  S') \geq ch$ for some $c > 0$. 
Then there exists $M_0 \in L^2(S', \mr^{2\times2}_{\sym})$ and $\zeta \in L^2(S', \stochsobolev_{\sym}^{1,2}(\Omega, \mr^2))$ with
\[
\sym \frac{  (R^h e_1, R^h e_2)^T \nabla' \widetilde u^h - \id_{2\times2}} h \wtwoscale M_0 + \sym \nabla_\omega \zeta.
\]
\end{lemma}
\begin{proof}
Using Theorem~\ref{thm:secorderdecomp} the proof is identical to \cite{Vel13}[Lemma~3.7].
\end{proof}

\subsubsection{Proof of Theorem~\ref{theoremlowbou}}
 \begin{proof}
Let $(u^h)$ be as in the claim, and let $S' \subset S$ be open with $C^{1,1}$ boundary, $(R^h)$, $(\widetilde R^h)$ and $\widetilde u^h$ be given as in Lemma~\ref{lemmacomp}.
Define $z^h$ by the decomposition of $u^h$ into 
\[
  u^h(x', x_3) = \overline u^h(x') + h x_3 \widetilde R^h(x')e_3 + hz^h(x',x_3),
\]
where once more $\overline u^h(x') = \int_I u^h(x',x_3) \td x_3$. Clearly we have $z^h \in W^{1,2}(S' \times I, \mr^3)$ with $\int_I z(x_3) \td x_3 = 0$. 

We define the approximate strain  
\[
  G^h := \frac { (R^h)^T \nabla_h u^h - \id_{3\times3}} h
\]
and split it into
\begin{equation}\begin{aligned}
\label{eq:Ghsplit}  
G^h &= \frac{ \iota \paren[\big]{(R^h e_1, R^h e_2)^T \nabla' \widetilde u^h - \id_{2\times2}}} h + \frac 1 h \sum_{\alpha = 1,2} (R^h e_3 \cdot \partial_\alpha \widetilde u^h) e_3 \otimes e_\alpha \\
       & \quad + \frac 1 h (R^h)^T ( \nabla' \overline u^h - \nabla' \widetilde u^h, 0) + \frac 1 h \paren[\Big]{ (R^h)^T \widetilde R^h e_3 \otimes e_3 - e_3 \otimes e_3 } \\
       & \quad + x_3 (R^h)^T (\nabla' \widetilde R^h e_3,0 ) + (R^h)^T \nabla_h z^h.
       \end{aligned}\end{equation}
Since $G^h$ is uniformly bounded in $L^2$, we may take a subsequence such that $G^h \wtwoscale G$ for some $G \in L^2(\Omega \times S \times I, \mr^{3\times 3})$. We study $\sym G$ by computing the possible two-scale limits of the terms in $\sym G^h$. For this we will readily take further subsequences if needed, without denoting them explicitly.

By applying Lemma~\ref{lem:cofortho} we obtain
\begin{equation*}\begin{aligned}
\sym \frac{ \iota \paren[\big]{(R^h e_1, R^h e_2)^T \nabla' \widetilde u^h - \id_{2\times2}}} h \wtwoscale  \iota(M_0 + \sym \nabla_\omega \zeta)
\end{aligned}\end{equation*}
for some $M_0 \in L^2(S', \mr^{2\times 2}_{\sym})$ and $\zeta \in L^2(S', \stochsobolev^{1,2}_{\sym}(\Omega, \mr^2))$. From Lemma~\ref{lem:twoscalegradient} and Prop.~\ref{propositionreasonable} we get

\[
  x_3 (R^h)^T ( \nabla' \widetilde R^h e_3) \wtwoscale x_3( \secf^u, 0)^T + x_3 R^T \nabla_\omega \phi^{\widetilde Re_3},
\]
as well as 
\[
\frac 1 h (R^h)^T( \nabla' \overline u^h - \nabla' \widetilde u^h) \wtwoscale R^T \paren[\Big]{\theta + \nabla_\omega v
}\]
for some $\theta \in L^2(S', \mr^{3 \times 2})$ and $v, \phi^{\widetilde R e_3} \in L^2(S', \stochsobolev^{1,2}(\Omega,\mr^3))$.

For 
\begin{equation*}\begin{aligned}
 \frac 1 h \paren[\Big]{ (R^h)^T \widetilde R^h e_3 \otimes e_3 - e_3\otimes e_3 }   + \frac 1 h \sum_{\alpha = 1,2} (R^h e_3 \cdot \partial_\alpha \widetilde u^h) e_3 \otimes e_\alpha\\
=
\frac 1 h \begin{pmatrix}
0                                       & 0                                       & \scalar{ R^h e_1, \widetilde R^h e_3 } \\
0                                       & 0                                       & \scalar{ R^h e_2, \widetilde R^h e_3 } \\
R^h e_3 \cdot \partial_1 \widetilde u^h & R^h e_3 \cdot \partial_2 \widetilde u^h & \scalar{ R^h e_3, \widetilde R^h e_3 } -1
\end{pmatrix}
\end{aligned}\end{equation*}
we obtain from Lemma~\ref{lem:(31)} and Lemma~\ref{lem:(33)} that
  \begin{equation*}\begin{aligned}
  &
    \sym \brackets*{\frac 1 h \paren[\Big]{ (R^h)^T \widetilde R^h e_3 \otimes e_3 - e_3 \otimes e_3 }   + \frac 1 h \sum_{\alpha = 1,2} (R^h e_3 \cdot \partial_\alpha \widetilde u^h) e_3 \otimes e_\alpha}
  \\
    & \wtwoscale
    \frac 1 \gamma \sym ( w^0 \otimes e_3) 
    +\frac { 1} { \gamma}\sym ( R^T \phi^{\widetilde R e_3} e_3 \otimes e_3) + \frac 1 \gamma \sym( R^T \nabla_\omega\phi^{\widetilde u}   e_3 \otimes e_3)
    \end{aligned}\end{equation*}
for some $w^0 \in L^2(S', \mr^3)$, $\phi^{\widetilde u} \in L^2(S', \stochsobolev^{2,2}(\Omega,\mr^3))$.
For the last term $(\nabla_h z^h)$ notice that \eqref{eq:Ghsplit} yields an uniform $L^2$ bound. By Lemma~\ref{tvrdnja2} we thus get
\[
  (R^h)^T \nabla_h z^h \wtwoscale  R^T \paren*{ \nabla_\omega \phi^z, \frac 1 \gamma D_{x_3}\phi^z }
\]
for some $\phi^z \in L^2(S', \stochsobolev^{1,2}(\Omega \times I,\mr^3))$.
We conclude that
\begin{align*}	
  \sym G^h \wtwoscale& \iota(M_0 + \sym \nabla_\omega  \zeta )
                + \frac 1 \gamma\sym ( w^0 \otimes e_3)
                +\frac { 1} { \gamma}\sym ( R^T \phi^{\widetilde R e_3} e_3 \otimes e_3) \\
                &+ \frac 1 \gamma \sym( R^T \nabla_\omega\phi^{\widetilde u}   e_3 \otimes e_3) 
                +\sym\paren[\Big]{R^T\theta + R^T\nabla_\omega v}\\
                &+x_3\sym\paren*{{{\iota(\secf^u)}}+ \paren[\big]{R^T (\nabla_\omega \phi^{\widetilde R e_3}), 0}}
                + \sym\paren[\Big]{R^T \paren[\big]{\nabla_\omega \phi^z, \frac 1 \gamma D_{x_3} \phi^z}}.
\end{align*}
We rewrite this as

\[
\sym G = \iota \paren[\Big]{\sym ( \widetilde B+ x_3 \secf^u)}  + \sym \paren*{ \nabla_\omega \phi, \frac 1 \gamma D_{x_3} \phi}, 
\]
where $\widetilde B = M_0 + [R^T \theta_{ij}]_{1 \leq i,j \leq 2}$ as well as 
\begin{align*}	
  \phi(x,\omega) := R^T(x') \widetilde \phi(x,\omega)+  \zeta(x',\omega) + \gamma x_3\begin{pmatrix} b_1 (x') & b_2(x') & 0 \end{pmatrix}^T,
\end{align*}
with
\begin{align*}	
\widetilde \phi(x,\omega) &=   \phi^z(x,\omega) +  v(x', \omega)  + x_3 \phi^{\widetilde R e_3}(x',\omega) + x_3 w_0(x') + \frac 1 \gamma \phi^{\widetilde u}(x',\omega), \\
b_i & =[R^T\theta(x')]_{3i} \quad \text{ for } i = 1,2.
       \end{align*}
Notice that $\phi \in \stochsobolev^{1,2}_{\sym,\gamma} (\Omega \times I, \mr^3)$.

After exhausting $S$ by $S' \subset S$ open with $C^{1,1}$ boundary, using Lemma~\ref{lemmacomp} and Remark~\ref{borism4} as well as the quadraticity of the form $\mathcal Q^{\gamma}$, the lower bound  follows by  using c from Assumption \ref{assumption} and   lower semicontinuity of the quadratic form $Q^0$ with respect to the stochastic two-scale convergence, see Lemma~\ref{lowersemicontinuity} and Remark \ref{borism6}  (see also \cite{Horneuvel12} for the details in the periodic case). 
  \end{proof}

\subsection{Proof of upper bound} 
In this section we prove the upper bound statement. We recall some issues from the periodic homogenization (see \cite{Horneuvel12}).  As in \cite{Schmidt-07} and other related results,
the key ingredient here is the density result for $W^{2,2}(S)$ isometric immersions
established in \cite{Hornung-arma2, Hornung-arma1} (cf.\ also \cite{Pakzad} for an earlier result in this direction).
It is the need for the results in \cite{Hornung-arma2} that forces us to restrict ourselves to domains $S$
which are not only Lipschitz but also piecewise $C^1$. More precisely, we only need that the outer unit normal
be continuous away from a subset of $\partial S$ with length zero.

For a given $\de\in W^{2,2}_{\textrm{iso}}(S)$ and for
a displacement $V\in W^{2,2}(S, \R^3)$ we denote by $q_V^\de$ the quadratic form
\[
q_V^\de = \sym \left((\nabla \de)^T(\nabla V)\right).
\]
We denote by $\mathcal A(S)$ the set of all
$\de\in W^{2,2}_{\textrm{iso}}(S, \R^3)\cap C^{\infty}(\overline{S}, \R^3)$
with the property that
\begin{align*}
\mathcal{S}:=\Big\{ B\in C^{\infty}(\overline{S}, \mathbf{R}^{2\times 2}_{\sym}) &: B = 0 \mbox{ in a neighborhood of }\{x'\in S : \secf^\de(x') = 0 \}\Big\}
\\
&\subset \{q_V^\de : V\in C^{\infty}(\overline{S}; \mathbf{R}^3)\}.
\end{align*}
In other words, if $\de\in\mathcal A(S)$ and $B\in C^{\infty}(\overline{S}, \mathbf{R}^{2\times 2}_{\sym})$ is a matrix field
which vanishes in a neighborhood of $\{\secf^\de = 0\}$, then there exists a displacement $V\in C^{\infty}(\o{S}; \mathbf{R}^3)$ such
that $q_V^\de = B$.
The necessary lemma for the proof of upper bound is the following lemma, whose proof is given in \cite{Schmidt-07,Horneuvel12}.
\begin{lemma}
	\label{le71}
	The set $\mathcal A(S)$ is dense in $W^{2,2}_{\textrm{iso}}(S)$ with respect to the strong $W^{2,2}(S,\mr^2)$ topology.
\end{lemma}
	\begin{proof}[Proof of Theorem \ref{theoremupperbound}]
	Fix some typical $\omega \in \Omega$. By Lemma~\ref{le71} it suffices to show the claim for $\de \in \mathcal{A}(S)$. Fix $B \in \mathcal{S}$  and  $V \in C^{\infty}(\overline{S}, \mathbf{R}^3)$ such that $q_{V}^{y}=B$,
and define the unit normal $n^u=\partial_1 \de \wedge \partial_2 \de$.
	Next we divide the domain $S$ into small cubes $(D_i^{\eta})_{i=1}^n$, $D_i^{\eta} \subset S$  of size $\eta$ such that $\abs*{ S \backslash \cup_{i=1}^n D_i^{\eta}} \to 0$ as $\eta \to 0$. On each cube we define $A_i^{\eta}, B_i^{\eta} \in \mathbf{R}^{2 \times 2}_{\sym}$ as the averages
\[
 A_i^{\eta}= \frac{1}{\abs{D_i^{\eta}}} \int_{D_i^{\eta}} \secf^{\de}(x')\td x', \quad  B_i^{\eta}= \frac{1}{\abs{D_i^{\eta}}} \int_{D_i^{\eta}} B(x')\td x'.
\]
For each $i=1,\dots,n$ and $\delta < \tfrac{\eta}{2}$ we define
\[
   D_i^{\eta,\delta}= \set{ x' \in D_i^{\eta}: \dist(x',\partial D_i^{\eta}) > \delta }. 
\]

For each $i  = 1, \ldots, n$  let $(g_{i}^{\eta,k}) \subset \left(\mathcal{C}^{\infty}(\Omega)\otimes C^{\infty} (I)\right)^3$ be a minimizing sequence of $\mathcal Q^\gamma$, in the sense that
\begin{align*}	
  &\int_{\Omega \times I} Q\paren[\big]{\omega,\iota(B^{\eta}_i +x_3 A_i^{\eta} )+\sym(D_1 g_{i}^{\eta,k}, D_2 g_{i}^{\eta,k}, \tfrac{1}{\gamma} D_{x_3} g_{i}^{\eta,k})} \td x_3 \drandommeasure - \frac 1 k \\
  &\leq \inf_{\phi \in W^{1,2}(\Omega \times I,\mr^3)} \int_{\Omega \times I} Q\paren[\big]{\omega,\iota(B^{\eta}_i +x_3 A_i^{\eta} )+\sym(D_1 \phi, D_2 \phi, \tfrac{1}{\gamma} D_{x_3} \phi)} \td x_3 \drandommeasure.
\end{align*}

We start with the Kirchhoff-Love ansatz, augmented by its linearization induced by the displacement $V$:
\begin{align*}
	v^h(x',x_3):=\de( x')+hx_3n^{u}( x') + h\paren[\big]{V(x') + h x_3\mu(x')},
\end{align*}
where $\mu$ is given by
\[
\mu = (I - n^{u}\otimes n^{u})(\d_1 V\wedge\d_2 \de + \d_1 \de \wedge\d_2 V).
\]
We set $R(x') = \paren[\big]{\nabla'\de(x'),\ n^{u}(x')}$. A straightforward computation
shows that
\begin{equation}
\label{vhers}
\nabla_hv^h = R + h\Big( (\nabla' V, \mu) + x_3(\nabla' n, 0) \Big) + h^2x_3(\nabla'\mu, 0).
\end{equation}

The actual recovery sequence $\de^h$ is obtained by adding to $v^h$ the oscillating correction of order $\eps = \eps(h)$:
\begin{equation}\label{eq:15}
\de^{\eta,k,\delta,h}(x',x_3) := v^h(x',x_3) + h\varepsilon \sum_{i=1}^n \chi_i^{\eta,\delta}(x') R(x')g_i^{\eta,k}(T_{\varepsilon^{-1} x'}\omega,x_3). 
\end{equation}
Here $\chi_{i}^{\eta,\delta} \in C^1(S)$ are smooth cut-off functions that satisfy
  \[
    |\chi_{i}^{\eta,\delta}| \leq 1; \quad \chi_{i}^{\eta,\delta} = 0 \text{ on } (D_i^{\eta} \backslash D_i^{\eta,\delta})^c \; \text{ and } \; |\nabla \chi_{i}^{\eta,\delta}| \leq \frac{C}{\delta} \text{ for some } C>0.
      \] 
Equations \eqref{vhers} and \eqref{eq:15} imply together with $n^u \cdot \mu = 0$ that
\begin{equation}
\label{fin-1}
\begin{split}
R^T\nabla_h \de^{\eta,k,\delta,h} = \id_{3\times3} &+ h\iota\left((\nabla' \de)^T(\nabla' V) + x_3\secf^\de\right)
+ h ( e_3 \otimes (\mu\cdot\nabla' \de),0)^T \\
&+ h (e_3\otimes (n\cdot\nabla' V),0)
\\
&+ h \sum_{i=1}^n \chi_i^{\eta,\delta} (D_1 g_i^{\eta,k},D_2 g_i^{\eta,k}, \tfrac{\eps}{h} D_{x_3} g_i^{\eta,k}) 
\\
&+ h^2x_3R^T(\nabla'\mu, 0) + h\varepsilon\sum_{i=1}^n \left(\chi_i^{\eta,\delta}(R^T\nabla' R)g_i^{\eta,k},0 \right)\\  &+ h\varepsilon\sum_{i=1}^n\left( \nabla'\chi_i^{\eta,\delta} g_i^{\eta,k},0 \right);
\end{split}
\end{equation}
the arguments of $g_i^{\eta,k}$ and of $(D_1 g_i^{\eta,k},D_2 g_i^{\eta,k},\frac \eps h  D_{x_3} g_i^{\eta,k})$ are $(T_{\varepsilon^{-1} x'}\omega,x_3)$.
  From \eqref{eq:15} and \eqref{fin-1} we conclude that
    \[
    \norm{ \de^{\eta,k,\delta,h}-\de     }_{W^{1,2}(S \times I, \mr^3)} \xrightarrow{h\to 0}0.   
      \]
Defining
\[
G^{\eta,k,\delta,h} = \frac{1}{h}\Big( R^T\nabla_h \de^{\eta,k,\delta,h} -\id_{3 \times 3}\Big)
\]
and using the fact that
$
n\cdot\nabla V + \mu\cdot\nabla' \de = 0,
$
we deduce from \eqref{fin-1} that
\begin{align*}
\sym G^{\eta,k,\delta,h} &=
\iota\left(q_V^\de + x_3\secf^\de\right)
\\
&+\sum_{i=1}^n \chi_i^{\eta,\delta}\sym (D_1 g_i^{\eta,k},D_2 g_i^{\eta,k}, \tfrac \eps h D_{x_3} g_i^{\eta,k})  + h\sym \left(x_3R^T(\nabla'\mu, 0) \right)\\&+ \varepsilon\sum_{i=1}^n \sym\left(\chi_i^{\eta,\delta}(R^T\nabla' R)g_i^{\eta,k},0 \right)\\  &+ \varepsilon\sum_{i=1}^n\sym\left( \nabla'\chi_i^{\eta,\delta} g_i^{\eta,k},0 \right);
\end{align*}
Using the objectivity property we obtain
\[
   W(\omega,x'/\varepsilon,\nabla_h \de^{\eta,k,\delta,h}) =W(\omega,x'/\varepsilon,I_{3\times3}+hG^{\eta,k,\delta,h}).
\]
It is also easy to see from \eqref{Taylorexpansion} that
\begin{equation*}
  \abs*{\frac{1}{h^2}W(\omega,x'/\varepsilon,\nabla_h \de^{\eta,k,\delta,h})-Q(\omega,x'/\varepsilon,G^{\eta,k,\delta,h})}\to 0,
\end{equation*}
uniformly in $x'$ for $h\to0$. It is not difficult to conclude  
\begin{align*}	
&\lim_{\eta \to 0}\lim_{k \to \infty}\lim_{\delta \to 0}\lim_{h \to 0} \frac{1}{h^2} \int_{S \times I}Q(\omega,x'/\varepsilon,G^{\eta,k,\delta,h})\td x' \td x_3 \\
& = \int_{S}  \inf_{\phi} \int_{\Omega \times I} Q\paren[\big]{\omega,\iota(B +x_3 \secf^{\de} )+\sym(D_1 \phi, D_2 \phi, \tfrac{1}{\gamma} D_{x_3} \phi)} \td x_3 \drandommeasure
\td x',
\end{align*}
where we minimize over $\phi \in \mathcal{W}_{\sym, \gamma}^{1,2}(\Omega \times I,\mr^3)$. 
By choosing appropriate $B$ we get
\[
   \lim_{\eta \to 0}\lim_{k \to \infty}\lim_{\delta \to 0}\lim_{h \to 0} \frac{1}{h^2} \int_{S \times I}Q(\omega,x'/\varepsilon,G^{\eta,k,\delta,h})\td x' \td x_3= \int_{S} \mathcal{Q}^{\gamma} (\secf^\de(x')) \td x'.
\]
The claim now follows by the lemma of Attouch and by a classical diagonal argument for $\Gamma$-convergence. 
\end{proof}  
\section{Examples for the probability space} 
The first example is the standard one and is already given in \cite{papanicolaou}. It covers the case of periodic homogenization. 
\begin{example}
	We take $W:\mathbf{R}^2 \times \mathbf{R}^{3 \times 3} \to [0,\infty]$ that is $1$-periodic in $x'$ and that satisfies the property $a,b,c$ from Assumption \ref{assumption}. Next we take $\Omega=T^2$ the $2$-dimensional unit torus with the Lebesgue measurable sets as the $\sigma$-algebra and the probability $P$ as Lebesgue measure on $T^2$. The measure is invariant under translation, e.g.,\ from the dynamical system $T_x \omega= \omega+x$ (mod $1$).
	The infinitesimal generators are the usual partial derivatives, for $i=1,2$
\[
   D_i= \frac{\partial}{\partial\omega_i} \ (\omega=(\omega_1,\omega_2)).
\]
In the end we define  $W(\omega,x',F)=W(\omega+x',F)$.  In this case we obtain for $\mathcal{Q}^{\gamma}$ the following formula
\begin{equation*} 
\mathcal{Q}^{\gamma} (G)= \inf_{\phi, B } \int_{T^2 \times I} Q\paren[\big]{x',\iota(B+x_3 G)+\sym(D_1 \phi, D_2 \phi, \tfrac{1}{\gamma} D_{x_3} \phi)} \td x' \td x_3,
\end{equation*}
where the infimum is taken over $\phi \in W^{1,2} (T^2 \times I, \mathbf{R}^3)$, $B \in \mathbf{R}^{2 \times 2}$.

\end{example}

In \cite{papanicolaou} it is also shown how almost periodic case can be covered with the abstract approach of stochastic homogenization.
Periodic (or almost periodic) homogenization naturally destroys the isotropic character of the energy density, even if the constituents are isotropic.  
In the next example we want to show how we can obtain the isotropic energy density out of isotropic constituents by stochastic homogenization. 
\begin{example}  \label{borism3} 
We would like to construct the probability space that consists of some subset of functions on $\mathbf{R}^2$ (taking values in some finite or countable  set) that is invariant under rotations of the coordinate system and moreover that the probability measure is invariant under these rotations. 
One possible construction can be made using Poisson processes.  
We  construct the probability space that consists of piecewise continuous functions. Namely, the construction of the probability space goes as follows: we take the Poisson point  process in $\mathbf{R}^2$ with the Lebesgue measure  and the sequence of independent random variables $(J_n)$ (independent of the Poisson process), taking values in the set $\mathbf{N}$.   
We then construct the marked Poisson point process, i.e., to every point (of some realization) we give a mark according to the realization of the sequence $(J_n)$.
For $i \in \mathbf{N}$ we take energy density functions $W^i:\mathbf{R}^2 \times \mathbf{R}^{3 \times 3} \to [0,\infty]$ that satisfies the properties $a,b,c$ from the Assumption \ref{assumption} and that is isotropic, i.e., it satisfies
\begin{align*}	
W^{i}(\omega,x',FR)= W^{i}(\omega,x',F), \textrm{ for all } F \in \mathbf{R}^{3 \times 3}, R \in \SO 3,\\
 \textrm{ for a.e. } (\omega,x') \in \Omega \times S \text{ and all } i \in \mathbf N.
\end{align*}
This implies that the same property is valid for the appropriate quadratic forms $Q^i$. 
 Out of each realization of the marked Poisson point process we make the material mixture in the following way: the point $x' \in \mathbf{R}^2$ is occupied by the material $i$ if the point $x'$ is in the Voronoi cell of the point that is marked with the number $i$. In that way we obtain the probability space where
\begin{eqnarray*}
\Omega=\{\textrm{piecewise constant functions that take values in the set $\mathbf{N}$ and \phantom{\hspace{1cm}}} \\ \textrm{that is constructed using marked Poisson process}   \}.
\end{eqnarray*}
For the $\sigma$-algebra we take the one generated by the sets 
\[
  \set[\big] { f \in \Omega  \ | \ f(x_l)=i_l, x_l \in \mathbf{Q}^2, i_l \in \mathbf{N}, \textrm{ for } l=1,\dots, n  }.
\]
 For the probability measure we take the pushforward of the measure given by the marked Poisson processes. The action $T_{x'}$ is simply given by the translation
$T_{x'} \omega (y')=\omega (x'+y')     $. It is easily seen that this action is measure preserving (since the distribution of the marked Poisson process is translation invariant) and ergodic. The ergodicity follows from the fact that the marked Poisson process (with independent marks) is ergodic and from the fact that the probability measure is just a pushforward. The energy density we define in the following way:
\[
  W(\omega,x',F)=W^i(\omega, x',F), \ \textrm{ if } \omega(x')=i. 
\]
For $R \in \SO 2$ we define the rotational transformation $R_t:\Omega \to \Omega$ in the following way
\[
  R_t(\omega)(x')=\omega(Rx'), \textrm{ for all } x' \in \mathbf{R}^2.
\]
Notice that $R_t$ is measure preserving. By $\widetilde R$ we denote the matrix in $\SO 3$ given by $\widetilde R=\iota(R)+e_3 \otimes e_3$. From the properties of the infinitesimal generator 
we easily conclude for $f \in \stochsmooth(\Omega)$ that
\[
  (D_1 (f\circ R_t), D_2(f \circ R_t)) =(D_1 f \circ R_t, D_2 f\circ R_t) R.
\]
From the cell formula \eqref{defQ} and the isotropy property of $Q^0$ we conclude
\begin{align*} 
\mathcal{Q}^{\gamma} (GR)&= \inf_{\phi, B } \int_{\Omega \times I} Q^0\paren[\big]{\omega,\iota(B+x_3 GR)+(D_1 \phi, D_2 \phi, \tfrac{1}{\gamma} D_{x_3} \phi)}  \drandommeasure \td x_3\\
&= \inf_{\phi, B } \int_{\Omega \times I} Q^0\paren[\big]{\omega,\iota(BR^T+x_3 G)+(D_1 \phi, D_2 \phi, \tfrac{1}{\gamma} D_{x_3} \phi)\widetilde R^T}  \drandommeasure \td x_3\\
&= \inf_{\phi, B } \int_{\Omega \times I} Q^0\paren[\big]{\omega,\iota(B+x_3 G)+(D_1 \phi, D_2 \phi, \tfrac{1}{\gamma} D_{x_3} \phi) }  \drandommeasure \td x_3 \\ 
&= \mathcal{Q}^{\gamma} (G).
\end{align*}
The infima are taken for $\phi \in W^{1,2}(\Omega \times I, \mr^3)$, $B \in \mathbf{R}^{2 \times 2}$.  This proves the isotropy. 
\end{example}
\begin{remark}
	Notice that in the first example we can write the cell formula in the following way
	\begin{equation*}  
	\mathcal{Q}^{\gamma} (G)= \inf_{\phi, B } \int_{\gamma T^2 \times I} Q^0\paren[\big]{x',\iota(B+x_3 G)+\sym(D_1 \phi, D_2 \phi,  D_{x_3} \phi)} \td x' \td x_3,
	\end{equation*}
	where the infimum is taken over $\phi \in W^{1,2} (\gamma T^2 \times I, \mathbf{R}^3)$, $B \in \mathbf{R}^{2 \times 2}$. In the second example  we can write  
	the cell formula in the following way 
    \[
         \mathcal{Q}^{\gamma} (G)= \inf_{\phi, B } \int_{ \Omega^d \times I} Q^0\paren[\big]{\omega,\iota(B+x_3 G)+(D_1 \phi, D_2 \phi,  D_{x_3} \phi)} \drandommeasure\td x_3,
    \]
	where the infimum is taken for $\phi \in W^{1,2}(\Omega^d \times I,\mathbf{R}^3)$, $B \in \mathbf{R}^{2 \times 2}$. Here, $\Omega^d$ is the transformed set of functions, where the transformation is given by
	\[
            \omega^d(x')=\omega(\gamma^{-1} x'), \textrm{ for every } \omega \in \Omega.           
        \]
	This transformation changes the intensity of the Poisson in the background. This explains the meaning of the parameter $\gamma$. Namely, although it seems that $\varepsilon(h)$ has not clear physical meaning, its meaning is incorporated in the probability space in the background, in the first case it is the size of the cell of the periodicity, while in the second it is connected to the intensity of the Poisson process in the background.  
\end{remark}

\def\cprime{$'$}

\appendix

\section{Appendix}
\subsection{Decompositions of $L^2$}
The decomposition of $L^2$ into a `gradient' part and a `divergence-free' part, known as Helmholtz-decomposition, is a classical result in real space, and has since been generalized in various aspects.

The aim in this section is to present possible decompositions of $L^2$, once a probability space is involved. 
From now on let $(\Omega, \mathcal{F}, P)$ be a probability space and an $N$-dimensional ergodic system $(T_x): \Omega \to \Omega$ (see Definitions \ref{defgroup} and \ref{defergodic}) as well as $D_i$, $i=1, \ldots N$ the infinitesimal generators of $T$ (see the discussion after Remark~\ref{rem:ergo-weak}.) We first recall a known result for $L^2(\Omega, \mr^N)$, we state a similar decomposition of $L^2(\Omega, \mr^{2 \times 2}_{\sym})$ into second gradients parts and a remainder, and finally we derive the decomposition for the `mixed' space $L^2( \Omega \times S, \mr^{N+M})$, where $S \subset \mr^M$.

\subsection{Decomposition of purely random spaces (first order)}
The Helmholtz-decomposition for stochastic $L^2$-spaces were already known. We recall the results given in \cite{gloria1}[Prop.~A.8, Prop.~A.9]:

Let 
\begin{align*}	
  L^2_{\mathrm{pot}}(\Omega) &:= \set { f \in L^2(\Omega, \mr^N): \int_\Omega (f_i D_j g - f_j D_i g) = 0,\\ 
                             &\hspace{4.5cm} \text{ for all } g \in W^{1,2}(\Omega), i,j = 1, \ldots, N}, \\
  L^2_{\mathrm{sol}}(\Omega) &:= \set { f \in L^2(\Omega, \mr^N): \int_\Omega f \cdot (\nabla_\omega  g) = 0, \quad\text{ for all } g \in W^{1,2}(\Omega)},
\end{align*}
as well as
\[
  F^2_{\mathrm{pot}}(\Omega) := \set*{ f \in L^2_{\mathrm{pot}}(\Omega) : \int_\Omega f = 0 }, \quad
  F^2_{\mathrm{sol}}(\Omega) := \set*{ f \in L^2_{\mathrm{sol}}(\Omega) : \int_\Omega f = 0 }.
\]
\begin{theorem}\label{thm:firstordercomp}
Let $(\Omega, \mathcal{F}, P)$ be a probability space and an $N$-dimensional ergodic system $(T_x): \Omega \to \Omega$. Then
\[
  L^2(\Omega, \mr^N) = F^2_{\mathrm{pot}}(\Omega)\oplus F^2_{\mathrm{sol}}(\Omega) \oplus \mr^N.
\]
Furthermore we can characterize the spaces by
\begin{align*}	
  F^2_{\mathrm{pot}}(\Omega) &= \adh_{L^2(\Omega,\mr^N)} \set { \nabla_\omega g: \quad g \in W^{1,2}(\Omega)},
\end{align*}
and for $N = 2,3$ we get respectively
\begin{align*}
  F^2_{\mathrm{sol}}(\Omega) &= \adh_{L^2(\Omega,\mr^{2})} \set { (-D_2, D_1) g: \quad g \in W^{1,2}(\Omega, \mr)},\\
  F^2_{\mathrm{sol}}(\Omega) &= \adh_{L^2(\Omega,\mr^{3})} \set {\nabla_\omega \times g: \quad g \in W^{1,2}(\Omega, \mr^3)}.
\end{align*}
\end{theorem}
\subsection{Decomposition of purely stochastic spaces (second order)}

For the second order decomposition we define the spaces 
\begin{align*}	
  L^2_{\mathrm{ppot}}(\Omega) := \set* { A \in L^2(\Omega, \mr^{2 \times 2}_{\sym}) \; \Big| \int_\Omega A : \cof \nabla_\omega  h = 0  \quad \text{ for all } h \in W^{1,2}(\Omega, \mr^2)},
\end{align*}
and 
\[
  L^2_{\mathrm{ssol}}(\Omega) := \set* { B \in L^2(\Omega, \mr^{2 \times 2}_{\sym})\; \Big | \int_\Omega B :\nabla^2_\omega h = 0  \quad \text{ for all } h \in W^{2,2}(\Omega, \mr)}.
\]
Denote further 
\[
  F^2_{\mathrm{ppot}}(\Omega) = \set*{ A \in L^2_{\mathrm{ppot}}(\Omega) \, \Big| \int_{\Omega} A = 0 };
 \; F^2_{\mathrm{ssol}}(\Omega) = \set*{ B \in L^2_{\mathrm{ssol}}(\Omega) \, \Big| \int_{\Omega} B = 0 }.
\]
We obtain the following decomposition and density result:
\begin{theorem}\label{thm:secorderdecomp}
Let $(\Omega, \mathcal{F}, P)$ be a probability space and a $2$-dimensional ergodic system $(T_x): \Omega \to \Omega$. Then
\[
  L^2(\Omega, \mr^{2\times2}_{\sym}) = F^2_{\mathrm{ppot}}(\Omega) \oplus F^2_{\mathrm{ssol}}(\Omega) \oplus \mr^{2\times2}_{\sym},
\]
as well as
\begin{equation*}\begin{aligned}\label{eq:Fssol}
F^2_{\mathrm{ppot}}(\Omega) &= \adh_{L^2(\Omega, \mr^{2\times2}_{\sym})} \set { \nabla^2_\omega b \;|\; b \in \stochsmooth(\Omega) },\\
F^2_{\mathrm{ssol}}(\Omega) &= \adh_{L^2(\Omega, \mr^{2\times2}_{\sym})} \set { \cof \sym \nabla_\omega b\; |\; b \in \stochsmooth(\Omega, \mr^2) }.
\end{aligned}\end{equation*}
\end{theorem}
The theorem can be proved by the same methods as in \cite{gloria1}.
\subsection{Decomposition of mixed spaces}
Let $S \subset \mr^M$ be a bounded Lipschitz domain, and let $L = M+N$. By $\nabla$ we will denote for maps with the domain $\mr^N \times S$ the operator $(\partial_1, \ldots, \partial_L)$ and for maps with the domain $\Omega \times S$ the operator $(D_1, \ldots, D_N, \partial_{N+1}, \ldots, \partial_L)$; since from the context the definition used is clear, we will not distinguish them in notation.
Furthermore in view of \eqref{eq:sobolevrealizations} after identification functions of the latter class with their realizations both operator coincide.   Additionally we define in both cases $\div = \nabla \cdot$.
\subsubsection{Trace Theorems}
In this section we briefly discuss a generalization of the trace operator for functions with the domain $\Omega \times S$.
The statements and proofs are analogous results to the classical results for Sobolev functions (see \cite{GR86}).

We define the extended trace
\begin{align*}	
\gamma: W^{1,2}(\randomspace \times S) \to L^2(\randomspace \times \partial S)\\
\gamma( \psi )(\randomelement, y) = \widetilde \gamma(\psi(\randomelement))(y)
\end{align*}
where
\[
  \widetilde \gamma : W^{1,2}(S) \to L^2(\partial S)
\]
is the classical trace.
It is easily seen that the map $\gamma$ is linear and continuous.
Furthermore the space $\gamma(W^{1,2}(\randomspace \times S))$ is a closed subspace of $L^2(\randomspace \times \partial S)$, which we will denote by
\[
  W^{1/2}(\randomspace \times \partial S) := \gamma\paren[\big]{W^{1,2}(\randomspace \times S)}.
\]
Together with the norm
\[
  \norm{ \mu }_{W^{1/2}(\randomspace \times \partial S) } = \inf_{\psi \in W^{1,2}(\randomspace \times S), \gamma(\psi) = \mu} \norm { \psi }_{W^{1,2}(\randomspace \times S)}
\]
the space is complete. To extend the trace, we define
on $ \stochsmooth(\Omega, C^\infty(\overline S))^L$ the norm 
\[
  \norm{ g }_{W^{1,2}_{\div}(\Omega \times S)}^2 = \norm {g}^2_{L^2(\Omega \times S)} + \norm { \div g}^2_{L^2(\Omega \times S)}
\]
and denote the completion of the space as 
\[
  W^{1,2}_{\div} ( \randomspace \times \twodomain) := \adh_{W^{1,2}_{\div}(\Omega \times S)} \; \stochsmooth(\Omega, C^\infty(\overline S))^L,
\]
analogously to the real-variant $W^{1,2}_{\div}(\mr^N \times \twodomain)$.

Furthermore we split $g = (g^s, g^r)$ into $g^s = (g_1, \ldots, g_N)$ and $g^r = (g_{N+1}, \ldots, g_L)$.

Let $\Gamma = \randomspace \times \partial \twodomain$ be the boundary of $\randomspace \times \twodomain$, and let $\nu$ be the outward-normal of $S$.
\begin{lemma}[Normal Trace Theorem]\label{lem:normaltrace}
The mapping $\gamma_\nu: g \mapsto \restrict{g^r}{\Gamma} \cdot \nu$ defined for $g \in  \stochsmooth(\Omega, C^\infty(\overline S))^L$ extends uniquely to a continuous, linear mapping $\gamma_\nu: W^{1,2}_{\div}( \randomspace \times \twodomain) \to (W^{1/2}(\randomspace \times \partial \twodomain))'$.
\end{lemma}
\begin{proof}
Using integration by parts for smooth functions $\varphi \in \stochsmooth(\Omega, C^\infty(\overline S)), g \in \stochsmooth(\Omega, C^\infty(\overline S))^L$ we have
  \begin{align*}	
  &\int_{\Omega \times S} g(\omega,x) \cdot  \nabla \varphi(\omega,x) \drandommeasure\td x +  \int_{\Omega \times S} (\div g)(\omega,x)  \varphi(\omega,x) \drandommeasure\td x \\
  &\hspace{6.3cm} = \int_{\Omega} \int_{\partial S} \nu(x) \cdot g^r(x, \omega) \varphi(x,\omega) \td x\, \drandommeasure.
  \end{align*}
 By continuity of the left-hand side in $g$ w.r.t.\ the $W^{1,2}_\div(\Omega \times S)$ topology 
there exists an extension, and by densitiy  the extension is unique.
\end{proof}

To simplify the notation we will write $\restrict{\psi}{\randomspace \times \partial \twodomain}$ for $\gamma(\psi)$ and $g^r\cdot \nu$ for $\gamma_\nu(g)$.

\subsubsection{Mixed differential equations}
 We define the spaces of test functions $\testfunctions = C^\infty_0(\mr^N \times \twodomain)$ and $X := C^\infty_0( \mr^N) \otimes C^\infty(\overline \twodomain)$, 
 and introduce for $f,g \in L^2_{\loc}(\mr^N, L^2(\twodomain))^{L}$ the notation
\begin{align*}	
  \nabla \times f = 0 \text{ in } \distributions  &\;:\Longleftrightarrow\; \int_{\R^N \times S} \paren[\big]{f_j \partial_i \varphi - f_i \partial_j \varphi }= 0 \, && \text{ for all } i,j \text{ and } \varphi \in \testfunctions,\\
  \div g = 0 \text{ in } \distributions &\;:\Longleftrightarrow\; \int_{\R^N \times S} \scalar[\Big]{g, \nabla \varphi } = 0 &&\text{ for all } \varphi \in  \testfunctions.
\end{align*} 

If additionally $g \in L^2_{\loc}(\mr^N, W^{1,2}_{\div}(S))$, then $ g^r\cdot \nu$ is almost everywhere well-defined and we set
\begin{align*}	
  \div g = 0 \text{ in } X' &\;:\Longleftrightarrow\; \int_{\R^N \times \R^M} \scalar[\Big]{g, \nabla \varphi } = \int_{ \mr^N } \int_{\partial\twodomain} \scalar { \varphi g^r, \nu }
                                                   \text{ for all } \varphi \in  X\matt.
\end{align*} 

Since  $\testfunctions \subset X$ the conditions immediately split into
\begin{align*}	
  \div g = 0 \text{ in } X' &\iff \div g  = 0\text{ in } \distributions \quad\text{ and }   \quad  g^{r}\cdot \nu  = 0 \quad \text{ a.e.\ on } \mr^N \times \partial S.
\end{align*}
Similar to before we define 
\[
   \widetilde \testfunctions = \stochsmooth( \randomspace ,  C^\infty_0( \twodomain))  \quad \text{ and }\quad \widetilde X = \set* { \psi \in \stochsmooth( \randomspace, C^\infty(\overline \twodomain)):  \int_{\randomspace \times \twodomain} \nabla \psi = 0}, 
\]
and for $f,g \in L^2(\Omega \times S, \mr^L)$ we denote
\begin{align*}	
    \nabla \times f = 0 \text{ in } \widetilde\distributions  &\;:\Longleftrightarrow\;\int_{\randomspace \times \twodomain} \paren[\big]{f_j \partial_i \psi - f_i \partial_j \psi }= 0 \, && \text{ for all } i,j \text{ and } \psi \in \widetilde\testfunctions,\\
    \div g = 0  \text{ in } \widetilde\distributions & \;:\Longleftrightarrow\; \int_{\randomspace \times \twodomain } \scalar { g, \nabla \psi } = 0 && \text{ for all } \psi \in \widetilde \testfunctions.
\end{align*}
Next we define for $g \in W^{1,2}_{\div}(\Omega \times S)$:
\[
    \div g = 0 \quad \text{ in } \widetilde X'  :\Longleftrightarrow \div g \text{ in } \widetilde {\testfunctions}' \quad \text{ and } g^r \cdot \nu = 0 \quad \text{ a.e.\ on } \Omega \times \partial S.
\]
Note that $\nabla \times f = 0$ in $\widetilde \testfunctions'$ holds, iff.\ almost all realization satisfy $\nabla \times f = 0$ in $\testfunctions'$, and similarly for $\div g$.

Finally define the sets
\begin{align*}	
  L^2_{\mathrm{pot}}(\Omega \times S) &= \set { f \in (L^2(\randomspace \times \twodomain))^{L}: \quad \nabla \times f = 0 \text{ in } \widetilde \testfunctions'},\\
  L^2_{\mathrm{sol}}(\Omega \times S)&= \set { g \in (L^2(\randomspace \times \twodomain))^{L}: \quad \div  g = 0 \text{ in } \widetilde X'},
\end{align*}
and
\begin{align*}	
F^2_{\mathrm{pot}}(\Omega \times S) &= \set*{ f \in L^2_{\mathrm{pot}}(\Omega \times S): \int_{\Omega \times S} f (\omega, x) \drandommeasure\td x= 0 }, \\
    F^2_{\mathrm{sol}}(\Omega \times S) &= \set*{ g \in L^2_{\mathrm{sol}}(\Omega \times S): \int_{\Omega \times S} g(\omega, x) \drandommeasure\td x = 0 }.
    \end{align*}

We are now able to state the decomposition theorem for the mixed-spaces:
\begin{theorem}\label{thm:weyl}
Let $(\Omega, \mathcal{F}, P)$ be a probability space {with $\mathcal F$ countable generated and} an $N$-dimensional ergodic system $(T_x): \Omega \to \Omega$ and let $S \subset \mr^M$ be a bounded Lipschitz domain.
 Then:
\begin{enumerate}[(i)]
\item 
 $ L^2(\Omega \times S) = F^2_{\mathrm{pot}}(\Omega \times S) \oplus F^2_{\mathrm{sol}}(\Omega \times S) \oplus \mr^L.$
\item 
$  F^2_{\mathrm{pot}}(\Omega \times S) = \adh_{L^2(\Omega\times S)} \set { \nabla g: \, g \in W^{1,2}(\Omega \times S)}.$
\item If $L =3$, i.e.\ $M = 1, N=2$ or $M = 2, N =1$ then
\begin{align*}
  F^2_{\mathrm{sol}}(\Omega \times S) &= \adh_{L^2(\Omega\times S,\mr^{L})} \set { \nabla \times g: \, g \in W^{1,2}(\Omega, \mr^{L})}.
\end{align*}
\end{enumerate}
\end{theorem}

\subsubsection{Orthogonality of $\div$ and $\nabla$}
The decomposition will follow easily, once we have proved the following lemma:
\begin{lemma}\label{lem:ortho}
  Let $f \in L^2_{\mathrm{pot}}(\Omega \times S)$ and $g \in L^2_{\mathrm{sol}}(\Omega \times S)$. Then
\begin{align*}	
&\int_{\randomspace \times \twodomain} f(\randomelement, y)\cdot g(\randomelement, y) \drandommeasure \td y \\
&\quad= \frac 1 { \abs S}\paren*{ \int_{\randomspace \times \twodomain} f(\randomelement, y) \drandommeasure \td y}\cdot \paren*{ \int_{\randomspace \times \twodomain} g(\randomelement, y) \drandommeasure \td y}.
\end{align*}
Especially if additionally $\int f = 0$ or $\int g = 0$, then $\scalar{ f,g}_{L^2(\Omega \times S)} = 0$.
\end{lemma}
Before proving the lemma, we first show that `multiplying' an ergodic system with a periodic system yields once more an ergodic system.
\begin{lemma}\label{lem:productdynamic}
Let $(\Omega, \mathcal{F}, P)$ be a probability space with an $N$-dimensional ergodic system $(T_x): \Omega \to \Omega$.
Let $\widetilde T : \mr^M \times [0,1)^M \to [0,1)^M$ be defined by $\widetilde T_y(\omega^M) = \omega^M + y \pmod 1$. Define the product dynamical system
\[
  T \times \widetilde T: (\mr^N \times \mr^M) \times ( \Omega \times [0,1)^M) \to \Omega \times [0,1)^M
\]
by
\[
  (T \times \widetilde T)_{(x,y)} (\omega^N, \omega^M) = (T_x \omega^N, T_y \omega^M ).
\]
If $T$ is ergodic, then $T \times \widetilde T$ is ergodic as well.
\end{lemma}
Here we use the weaker formulation of ergodicity, explained in Remark~\ref{rem:ergo-weak}.
\begin{proof}
Let $B \subset \Omega \times [0,1)^M$ be measurable and invariant under $T \times \widetilde T$, i.e.
\[
  (T \times \widetilde T)_{(x,y)}(B) = B \quad \text{ for all } (x,y) \in \mr^N \times \mr^M.
\]
Choosing $x = 0$ we get
\[
 \bigunion_{y \in \mr^M} (T \times \widetilde T)_{(0, y)}(B) = B.
\]
Thus $B$ can be written as $B' \times [0,1)^M$ with $B' \subset \Omega$ measurable. We have
\[
  (\randommeasure \otimes \mathscr{L}^M)(B) = \randommeasure(B').
\]
By ergodicity of $T$ we have $\randommeasure(B') \in \set { 0,1}$ and thus $(\randommeasure \otimes \mathscr{L}^M)(B) \in \set {0,1}$.
\end{proof}
\begin{proof}[Proof of Lemma~\ref{lem:ortho}]
By translating and scaling it suffices to show it for domains $\twodomain \subset Q := \frac 1 2 (-1,1)^M$. 
Fix some $f,g \in L^2(\randomspace \times \twodomain)$ and extend them to $f,g \in L^2(\randomspace \times Q)$ with their corresponding mean-value on $\randomspace \times \twodomain$. Finally extend both function $Q$-periodic onto $\mr^M$.
Assuming $\nabla \times f = 0$ in $\distributions(\randomspace \times \twodomain)$ and $\nabla \cdot g = 0$ in $\widetilde X'(\randomspace \times \twodomain)$, the extended functions
satisfy the PDE clearly on $\randomspace \times (\mz^M + \twodomain)$. For some $\omega \in \Omega$ typical we define the sequence of functions
\begin{equation}\begin{aligned}
\label{eq:def_feps_geps}
  f^\eps(x, y) = f(T_{\eps^{-1} x}\omega, \eps^{-1} y), \quad
  g^\eps(x, y) = g(T_{\eps^{-1} x}\omega, \eps^{-1} y).
\end{aligned}\end{equation}

We prove this lemma by showing that 
\begin{equation}\begin{aligned}\label{eq:separate_convergence}
f^\eps \cdot g^\eps \weakstar \overline f \cdot \overline g \quad \text{ in } \distributions(\mr^N \times Q),
  \end{aligned}\end{equation}
where $\overline f,\overline g$ are the corresponding weak limits for $f^\eps, g^\eps$ as defined in \eqref{eq:def_feps_geps}. By Birkhoff's Theorem we have
\[
  \overline f = \frac 1 {\abs S} {\int_{\randomspace \times \twodomain} f(\randomelement, z) \drandommeasure \td z }, \quad
  \overline g = \frac 1 {\abs S} {\int_{\randomspace \times \twodomain} g(\randomelement, z) \drandommeasure \td z },
\]
as well as
\begin{equation}\begin{aligned}\label{eq:weakproduct}
f^\eps \cdot g^\eps \weakstar \overline {fg}\quad \text{ in } L^1_{\loc}(\mr^N \times Q)
  \end{aligned}\end{equation}
by Lemma~\ref{lem:productdynamic}. By uniqueness of the limit both have to agree, which is the claim of the lemma.

From \eqref{eq:weakproduct} we deduce, that convergence holds for every $\eps \to 0$, and to identify the limit in terms of $\overline f, \overline g$ it suffices to choose the specific sequence $\eps_n = n^{-1}$, where we suppress the index $n$ and still write $\eps \to 0$ instead of $n\to\infty$.

To show \eqref{eq:separate_convergence} we simulate the div-curl lemma (see e.g., \cite{Allaire-02}[Lemma~1.3.1]). By locality of the statement we can reduce ourselves to the case $K \subset\subset \mr^N \times Q$ and define $K^\twodomain = K \intersect ( \mr^N \times \twodomain )$. We can assume that $K^S$ has Lipschitz boundary. Furthermore we may assume that the weak limits of $f^\eps, g^\eps$ are zero. 

Define $\psi$ to be the primitive of $f$ on the domain $\mr^N \times  \twodomain$ with the property $\int_{K^\twodomain }  \psi(x,  y) = 0$. Extend $\psi$ onto $\mr^N \times ( \mz^M + \twodomain)$ periodically.
Furthermore  define 
\[
  \psi^\eps(x,y) = \eps \psi( \frac x \eps, \frac y \eps ) + c^\eps \quad \text{ on } \quad \mr^N \times \eps(\mz^M + \twodomain),
\]
with the constant
\[
  c^\eps = - \int_{K^\twodomain} \eps \psi(\frac x \eps ,  y) \td (x,y).
\]
By construction we have
\begin{align*}	
\nabla \psi^\eps = f^\eps \text{ on } \mr^N \times \eps(\mz^M +S) \quad \text{ as well as } \quad \int_{K^\twodomain}  \psi^\eps(x, \eps y)  \td (x,y)= 0.
\end{align*}

We define the finite set $Z_\eps := \mz^M \intersect  (-\eps^{-1},\eps^{-1})^M$ and partition $Q$ (up to a null set) into 
\[
  Q = \bigunion_{k \in Z_\eps} \eps( Q + k).
\]
Fix some $\varphi \in C^\infty_0(\mr^n \times Q)$ and we compute
\begin{align*}	
  \int_{\mr^n \times Q} \varphi(x,y) \td(x,y)  & = \sum_{k \in Z_\eps} \int_{\mr^n \times (\eps Q + \eps k)} \varphi(x,y) \td(x,y)\\
                                                               &\quad=  \int_{\mr^n \times (\eps Q )} \brackets*{\sum_{k \in Z_\eps}\varphi(x,y+\eps k)} \td(x,y).
\end{align*}
Using additionally the periodicity of $f^\eps, g^\eps$, we thus get
\begin{align*}	
  &\int_{\mr^N \times Q} f^\eps(x,y)\cdot g^\eps(x,y)  \varphi(x,y)\td(x,y) \\
  &=  
\int_{\mr^N \times ( \eps Q) } f^\eps(x,y) \cdot g^\eps(x,y) \brackets*{\sum_{k \in Z_\eps}\varphi(x, y+\eps k)}   \td(x,y)\\
  &=  \int_{\mr^N \times \brackets{ \eps (Q\setminus \twodomain)} } f^\eps(x,y) \cdot g^\eps(x,y) \brackets*{\sum_{k \in Z_\eps}\varphi(x, y+\eps k)}  \td(x,y) \\
&\quad +\int_{\mr^N \times  (\eps \twodomain) } f^\eps(x,y) \cdot g^\eps(x,y) \brackets*{\sum_{k \in Z_\eps}\varphi(x, y+\eps k)}  \td(x,y) \\
&=\int_{\mr^N \times  [\eps (Q\setminus \twodomain)] } \overline {f} \cdot \overline{g}\cdot  \brackets*{\sum_{k \in Z_\eps}\varphi(x, y+\eps k)}  \td(x,y) \\
&\quad+ \int_{\mr^N \times  \partial(\eps \twodomain) } \scalar{ (g^r)^\eps(x,y), \nu}\brackets*{\psi^\eps(x,y) \sum_{k \in Z_\eps}\varphi(x, y+\eps k)} \td(x,y) \\
&\quad- \int_{\mr^N \times  (\eps \twodomain) } \psi^\eps(x,y) \cdot  \div\brackets*{g^\eps(x,y)\sum_{k \in Z_\eps}\varphi(x, y+\eps k)}  \td(x,y).
\end{align*}
The first term vanishes by the assumption on the weak limits of $f^\eps, g^\eps$, while the second one vanishes by the boundary condition on $g$. In the last term we apply the product rule: the term, where the divergence falls on $g^\eps$ vanishes, by using the PDE and density.

We are left with
\begin{align*}
  &\int_{\mr^N \times Q} f^\eps(x,y)\cdot g^\eps(x,y)  \varphi(x,y) \td (x,y)\\
&= -\int_{\mr^N \times (\eps \twodomain)}  \psi^\eps (x,y) \cdot  \scalar*{ g^\eps(x,y), {\sum_{k \in Z_\eps}\nabla\varphi(x, y+ \eps k)} } \td (x,y) \\
&= -\int_{K^\twodomain}  \psi^\eps (x,\eps y) \cdot  \scalar*{ g^\eps(x,\eps y), {\eps^M\sum_{k \in Z_\eps}\nabla\varphi(x,\eps  y+ \eps k)} }\td (x,y).
\end{align*}
We note that $(x,y) \mapsto \psi^\eps (x,\eps y)$ is uniformly bounded in $L^2(K^\twodomain)$. Indeed, using Poincar\'e's inequality, recalling $ \int_{K^\twodomain} \psi^\eps(x, \eps y) \td(x,y) = 0$, yields
\[
  \norm{ \psi^\eps (x, \eps y)  }_{L^2(K^\twodomain)} \leq  C_{K^\twodomain} \norm{ f^\eps(x, \eps y)}_{L^2(K^\twodomain)} = C_{K^\twodomain} \norm{ f(T_{\eps^{-1}x} \randomelement,  y)}_{L^2(K^\twodomain)}.
\]
Furthermore the sequence 
\[
 (x,y) \mapsto  f^\eps( x, \eps y) = f(T_{\eps^{-1}x} \randomelement, y)
\]
is uniformly bounded in $L^2(K^S)$ for almost all $\randomelement \in \randomspace$. To see this define for every $x \in \mr^N$ the  cross sections
\[
  K^\twodomain_x := \set *{y \in \mr^M :   (x,y) \in K^\twodomain } \subset \twodomain,
 \]
and thus
\begin{align*}	
\int_{K^\twodomain} \abs { f(T_{\eps^{-1}x} \randomelement, y)}^2 \td (x,y)
  &=  \int_{\mr^N} \int_{K^\twodomain_x}\abs { f(T_{\eps^{-1}x} \randomelement, y)}^2   \td y \td x\\
  &\leq  \int_{ \set {x \in \mr^N: K^S_x \neq \emptyset}}\paren*{ \int_{\twodomain}\abs { f(\cdot, y)}^2   \td y }(T_{\eps^{-1}x} \randomelement)\td x.
  \end{align*}
By the Ergodic Theorem the integrand converges for almost every $\randomelement$ to $C \norm{ f }_{L^2(\randomspace\times S)}^2$, for some constant $C > 0$ depending only on $K^S$.
Thus for almost every $\randomelement \in \randomspace$ we have
\[
\limsup_{ \eps \downarrow 0 } \int_{K^\twodomain} \abs { f(T_{\eps^{-1}x} \randomelement, y)}^2 \td (x,y) < \infty,
\]
and therefore the left-hand side is uniformly bounded for almost every $\randomelement$.

Noticing also that
\[
  \nabla \paren* { (x,y) \mapsto  \psi(x, \eps y) } = (f_1, \ldots, f_N, \eps f_{N+1}, \ldots, \eps f_L),
\]
we have a uniform bound on $(x, y) \mapsto \psi^\eps (x, \eps y)$
in $W^{1,2}(K^\twodomain)$ and thus a subsequence converging weakly to some $\Psi \in W^{1,2}(K^\twodomain)$. By Rellich's Theorem we have also strong convergence in $L^2(K^S)$. Additionally $\Psi$ does not depend on $y$: to see this, we apply Poincar\'e's inequality once more and obtain
\begin{align*}	
\norm{ \partial_y \Psi }_{L^2(K^S)} &\leq \liminf_{\eps\downarrow 0} \norm{ \partial_y ( (x,y) \mapsto \psi^\eps(x,\eps y)) }_{L^2(K^S)}\\
  & \leq    \liminf_{\eps\downarrow 0}\eps\norm{ (x,y) \mapsto f^\eps(x,\eps y) }_{L^2(K^S)} = 0.
\end{align*}

The sequence of functions $(x,y) \mapsto g^\eps(x, \eps y) = g(T_{\eps^{-1} x}\omega, y)$ converges weakly to $y \mapsto (\int_{\randomspace} g(\omega,y) \td \omega)$ in $L^2(Q)$, a function independent of $x$. Finally observe that
\[
\brackets*{ (x,y)\mapsto \eps^M\sum_{k \in Z_\eps}\nabla\varphi(x,\eps  y+ \eps k)} 
 \to
\brackets[\Big]{x \mapsto \int_Q \nabla \varphi(x, \hat y)\td \hat y}
\]
uniformly in $x$. We thus have
\begin{align*}	
 &\int_{\mr^N \times  \twodomain}  \scalar*{ g^\eps(x,\eps y),  \psi^\eps (x,\eps y)  \paren[\Big]{\eps^M\sum_{k \in Z_\eps}\nabla\varphi(x,\eps  y+ \eps k)} }\td(x,y)  \\
 & \to \int_{\mr^N \times  \twodomain}  \scalar*{ { \int_{\randomspace} g(\omega,y) \drandommeasure}\;,\;  \Psi (x)  {\int_Q \nabla \varphi(x, \hat y) \td \hat y} }\td(x,y),
\end{align*}
since the first factor converges weakly and the second strongly. Rearranging the integrals yields
\begin{align*}	
 &\int_{\mr^N \times  \twodomain}  \scalar*{ { \int_{\randomspace} g(\omega,y) \drandommeasure}\;,\;  \Psi (x) \int_Q \nabla \varphi(x, \hat y) \td \hat y } \td(x,y)\\
 &\hspace{3cm} = \scalar*{ \abs{\twodomain}\cdot \overline g \;,\;  \int_{\mr^N \times Q} \Psi (x) \cdot  \nabla \varphi(x, \hat y) \td \hat y \td x } = 0,
\end{align*}
since $\overline g = 0$. This finishes the proof.
\end{proof}

\subsubsection{The decomposition}
We will prove Theorem~\ref{thm:weyl} (i) similar to \cite{zhikov2}[Lemma~7.3]. For this we introduce a mollifier in the mixed setting.
Let $K_1 \in C^\infty_0(\mr^N), K_2 \in C^\infty_0(\mr^M)$ be standard mollifier, i.e. $K_1, K_2$ are even functions, i.e. $K_i(x) = K_i(-x)$ for all $x$ and $i=1,2$, and
\[
  K_1, K_2 \geq 0, \quad \int_{\mr^N} K_1 = \int_{\mr^M} K_2 = 1.
\]
Define for $\delta > 0$ the sequences
\[
  K_1^\delta(s) = \frac 1 {\delta^N} K_1( \delta^{-1} s), 
\quad
  K_2^\delta(y) = \frac 1 {\delta^M} K_2( \delta^{-1} y),
\]
and further the mollification-operators $\mathcal J^\delta$ for $g \in L^2(\Omega \times \mr^M)$ by
\[
  (\mathcal J^\delta g)(\omega,x) = \int_{\mr^N} \int_{\mr^M} K_1^{\delta}(s) K_2^{\delta}(x-y) g( T_s \omega, y) \td y \td s.
\]
It is easily seen that $\mathcal J^\delta g$ is a continuous, linear, symmetric operator $L^2(\Omega \times \mr^M) \to L^2(\Omega \times \mr^M)$ with
\[
  \lim_{\delta \downarrow 0} \mathcal J^\delta g = g \quad \text{ strongly in } L^2(\Omega \times \mr^M).
\]
Furthermore $\mathcal J^\delta g \in \stochsmooth(\Omega, C^\infty(\mr^M))$ for $g \in L^2(\Omega \times \mr^M)$ and 
\[
  \nabla (\mathcal J^\delta g) = \mathcal J^\delta \nabla g \quad \text { for all } g \in W^{1,2}(\Omega \times \mr^M).
\]

\begin{proof}[Proof of Theorem~\ref{thm:weyl} (i)]
The orthogonality between $F^2_{\mathrm{pot}}(\randomspace \times \twodomain)$ and $L^2_{\mathrm{sol}}(\randomspace \times \twodomain)$ follows from Lemma~\ref{lem:ortho}. Therefore 
\[
 L^2_{\mathrm{sol}}(\randomspace \times \twodomain) \subset \brackets{F^2_{\mathrm{pot}}(\randomspace \times \twodomain)}^\perp.
\]
For the reverse inclusion let $g \in \brackets{F^2_{\mathrm{pot}}(\randomspace \times \twodomain)}^\perp$, i.e.
\begin{equation*}\label{eq:ortho}\begin{aligned}
\scalar* { g, f }_{L^2} = 0 \quad \text{ for all } f \in F^2_{\mathrm{pot}}(\randomspace \times \twodomain).
\end{aligned}\end{equation*}
Fix some $\varphi \in \widetilde \testfunctions$, and note that $\nabla \varphi \in F^2_{\mathrm{pot}}(\randomspace \times \twodomain)$. 
Extend both $\varphi$ and $g$ by $0$ to functions defined on $\Omega \times \mr^M$

Fix some $\delta > 0$. By using the elementary properties of the mollification operator $\mathcal J^\delta$ we have 
\[
  0 = \scalar { g , \nabla \mathcal J^\delta \varphi }_{L^2}  =\scalar { g , \mathcal J^\delta(\nabla \varphi)}_{L^2}  =\scalar{ \mathcal J^\delta g, \nabla \varphi }_{L^2} 
                                         = -\scalar{ \div \mathcal J^\delta g, \varphi}_{L^2}.
\]
 By the density of $\widetilde \testfunctions \subset \widetilde X \subset L^2( \randomspace \times \twodomain)$ we get $\div (\mathcal J^\delta g) = 0$ a.e., thus
\[
   \div \mathcal J^\delta g = 0 \quad \text{ in } \widetilde \distributions.
\]
By the strong convergence $\mathcal J^\delta g \to g$ in $L^2(\Omega \times S)^L$ we get
\[
   \div g = 0 \quad \text{ in } \widetilde \distributions.
\]
Furthermore for any $\psi \in \widetilde X$ we have $\nabla \psi \in F^2_{\mathrm{pot}}$ as well and thus
\begin{equation*}\label{eq:intbyparts}\begin{aligned}
0 = \int_{\randomspace \times \twodomain } \scalar{ \mathcal J^\delta g, \nabla \psi }
&= \int_{\randomspace \times \twodomain } \div ( \mathcal J^\delta (g)\cdot \psi)  \\
&= \int_{\randomspace \times \twodomain } \sum_{k=N+1}^L \partial_k(\mathcal J^\delta (g)_k\cdot \psi)
 = \int_{\randomspace}\int_{\partial \twodomain} \scalar{ \psi \mathcal J^\delta(g^r), \nu}.
  \end{aligned}\end{equation*}
Note that $\mathcal J^\delta g \to g$ in $L^2(\Omega \times S)$ together with $\div g = \div \mathcal J^\delta g =0$ implies that $\mathcal J^\delta g  \to g$ in $W^{1,2}_{\div}(\Omega \times S)$.
Together with the equality
\[
  \int_{\randomspace}\int_{\partial \twodomain} \scalar{ \psi \mathcal J^\delta(g''), \nu} = 0,
\]
following from Lemma~\ref{lem:normaltrace}, the strong convergence $\mathcal J^\delta g\to g$ in $W^{1,2}_{\div}(\Omega \times S)$ is enough to conclude that $\div g = 0$ in $\widetilde X'$, thus $g \in L^2_{\mathrm{sol}}(\Omega \times S)$.
\end{proof}
For the proof of Theorem~\ref{thm:weyl} (ii) we follow \cite{gloria1}:

\begin{proof}[Proof of Theorem~\ref{thm:weyl} (ii)]
From classical Hilbert space theory follows
\[
L^2(\randomspace \times \twodomain) = \adh_{L^2} \set { \nabla \chi : \chi \in \widetilde X } \oplus \brackets*{\adh_{L^2} \set { \nabla \chi : \chi \in \widetilde X }}^\perp.
\]
By the previous orthogonal decompositions it is enough to show that
\[
\brackets*{\adh_{L^2} \set { \nabla \chi : \chi \in \widetilde X }}^\perp = L^2_{\mathrm{sol}}(\randomspace \times \twodomain),
\]
since then
\[
\adh_{L^2} \set { \nabla \chi : \chi \in \widetilde X } = F^2_{\mathrm{pot}}(\randomspace \times \twodomain)
\]
follows trivially from 
\[
    L^2(\randomspace \times \twodomain) = F^2_{\mathrm{pot}}(\randomspace \times \twodomain)\oplus L^2_{\mathrm{sol}}(\randomspace \times \twodomain).
\]
But 
\[
\brackets*{\adh_{L^2} \set { \nabla \chi : \chi \in \widetilde X }}^\perp = L^2_{\mathrm{sol}}(\randomspace \times \twodomain)
\]
was just the definition of the space $L^2_{\mathrm{sol}}(\Omega \times S)$.
\end{proof}
The claim of Theorem~\ref{thm:weyl} (iii) can be proven almost identically.



{\bf Acknowledgment.} 
The first two authors were fully supported by DFG. 
This third author has been fully supported by Croatian Science Foundation grant number 9477.

\end{document}